\documentclass{amsart}
                      
\usepackage{amsmath}
\usepackage{amssymb}
\usepackage{amsthm}
\usepackage{leftidx} 

\usepackage[utf8x]{inputenc}
\usepackage[english]{babel}

\usepackage{vmargin} 
\usepackage{graphicx} 
\usepackage{array} 
\usepackage{multirow}
\usepackage{xypic} 

\usepackage{enumerate}

\usepackage[colorlinks=false, breaklinks=true, linktocpage=true, pdfborder={0 0 0}]{hyperref}

\usepackage{microtype}

\theoremstyle{plain} 
\newtheorem{Thm}{Theorem}[section] 
\newtheorem{Cor}[Thm]{Corollary}
\newtheorem{Prop}[Thm]{Proposition}

\newtheorem{Lemma}[Thm]{Lemma}
\newtheorem{Quest}[Thm]{Question}
\newtheorem*{Thm*}{Theorem}

\newtheorem{IThm}{Theorem}
\newtheorem*{IConj}{Conjecture}

\theoremstyle{definition}
\newtheorem{Def}[Thm]{Definition}
\newtheorem{Deff}[Thm]{Definitions}

\theoremstyle{remark}
\newtheorem{Rem}[Thm]{Remark}
\newtheorem{Remm}[Thm]{Remarks}

\newtheorem{Ex}[Thm]{Example}
\newtheorem{Exx}[Thm]{Examples}
\newtheorem*{Notation}{Notation}

\newcommand{\C}{\mathbb{C}}
\newcommand{\N}{\mathbb{N}}
\newcommand{\R}{\mathbb{R}}

\newcommand{\A}{\mathbb{A}}
\newcommand{\B}{\mathbb{B}}
\newcommand{\PR}{\mathbb{P}}

\newcommand{\Hom}{\operatorname{Hom}}

\newcommand{\id}{\operatorname{id}}

\newcommand{\Sp}{\operatorname{Sp}}

\newcommand{\category}{\operatorname} 
\newcommand{\rigVar}{\mathit{smRig}}

\newcommand{\tensor}{\hat{\otimes}}

\newcommand{\sSet}{\mathbf{S}}
\newcommand{\vect}{\operatorname{Vect}}
\newcommand{\Vect}{\mathcal{V}}
\newcommand{\sPSh}{\mathit{sPSh}}

\title{Homotopy Classification of Line Bundles Over Rigid Analytic Varieties}
\author{Helene Sigloch} 
\address{Helene Sigloch\\
        Fakult{\"a}t f{\"u}r Mathematik\\
        Universit{\"a}t Regensburg\\
        93040 Regensburg\\
        Germany}
\email{helene.sigloch@ur.de}
\thanks{The author was supported by the Brigitte Schlieben-Lange-Programm of the German state Baden-W{\"u}rttemberg 
  and by the DFG in the GK 1821.}  
                    
\subjclass[2010]{Primary 14F22, 14G22; Secondary 55R15}

\begin{document}
\begin{abstract}
 We construct a motivic homotopy theory for rigid analytic varieties
 with the rigid analytic affine line $\mathbb{A} ^1_\mathrm{rig}$ as an interval object.
 This motivic homotopy theory is inspired from, but not equal to, Ayoub's 
 motivic homotopy theory for rigid analytic varieties.
 Working in the so constructed homotopy theory,
 we prove that a homotopy classification of vector bundles of rank~$n$ 
 over rigid analytic quasi-Stein spaces follows from $\mathbb{A}^1_\mathrm{rig}$-homotopy invariance
 of vector bundles. This $\mathbb{A}^1_\mathrm{rig}$-homotopy invariance
 is equivalent to a rigid analytic version of Lindel's solution to the Bass--Quillen conjecture.
 Moreover, we establish a homotopy classification of line bundles over rigid analytic quasi-Stein spaces.
 In fact, line bundles are classified by infinite projective space.
\end{abstract}
\maketitle

\section{Introduction}
 The goal of this article is to develop tools to classify vector bundles over suitable rigid analytic varieties in a
 suitable homotopy theory. Our main results are:
 \begin{IThm}[Theorem~\ref{ThmClassLines}]\label{ThmClassIntro}
  Let $k$ be a complete, non-archimedean, non-trivially valued field.
  Let $X$ be a smooth $k$-rigid analytic quasi-Stein variety. Then there is a natural bijection
  \begin{equation*}
   [X, \mathbb{P}^\infty ]_{\mathbb{A}^1_\mathrm{rig}} \overset{\sim}{\longrightarrow} 
   \{ \text{analytic $k$-line bundles over $X$} \} / \cong 
  \end{equation*}
  between motivic homotopy classes from $X$ into the infinite projective space $\mathbb{P}^\infty$ over $k$
  and isomorphism classes of analytic $k$-line bundles over $X$.
  The homotopy theory is defined in section~\ref{sechotheo}.
 \end{IThm} 
 
 The theorem follows from the following more general result.
 \begin{IThm}[Theorem~\ref{ThmClass}]\label{ThmAHWIntro}
  Let $\mathcal{R}$ be a subcategory of the category of smooth rigid analytic varieties. Assume that $\mathcal{R}$ 
  has a strictly initial object $\emptyset$. 
  Let $I \in \{ \mathbb{B}^1 , \mathbb{A}^1_\mathrm{rig} \}$ be a representable interval object on $\mathcal{R}$.
  For some $n \in \mathbb{N}$, assume that vector bundles of rank $n$ are $I$-invariant on $\mathcal{R}$.
  Assume that every vector bundle of rank $n$ over an object $X$ of $\mathcal{R}$ admits a finite local
  trivialisation. Then there is a natural bijection
  \begin{equation*}
   [X, \operatorname{Gr}_{n/k} ]_{I} \overset{\sim}{\longrightarrow} 
    \{ \text{analytic $k$-vector bundles of rank $n$ over $X$} \} / \cong .
  \end{equation*}
 \end{IThm}
 
 Both theorems can be seen in the context of Morel's $\mathbb{A}^1$-homotopy classification of algebraic vector bundles
 over a smooth affine variety and of Grauert's homotopy classification of holomorphic vector bundles over a complex Stein space. 
 The classical model for all of these theorems is 
 Steenrod's homotopy classification of continuous vector bundles over a paracompact Hausdorff space.
 \begin{IThm}[F.~Morel \cite{Morel12}, Asok--Hoyois--Wendt \cite{AsokHoyoisWendt15}]
  Let $X$ be a smooth affine algebraic variety over a field $K$. There is a natural bijection
  \begin{equation*}
   [X, \mathrm{Gr}_{n/K} ]_{\mathbb{A}^1} \overset{\sim}{\longrightarrow} \{ \text{alg.~$K$-vector bundles of rank $n$ over $X$} \} / \cong 
  \end{equation*}
  between $\mathbb{A}^1$-homotopy classes from $X$ into the infinite Grassmannian $\mathrm{Gr}_{n}$ over $K$
  and isomorphism classes of algebraic $K$-vector bundles of rank $n$ over $X$.
 \end{IThm}
 If $n=1$, the theorem holds more generally for any smooth scheme $X$ over a regular base \cite[§4, Proposition~3.8]{MV}.
  
 Strikingly, if $K = \mathbb{C}$, this classification still holds after analytification.
 The analytification of a smooth affine complex algebraic variety is a Stein manifold and the analytification
 of the complex affine line is $\mathbb{C}$.
 
 \begin{IThm}[Grauert \cite{Grauert57a,Grauert57b}]
  Let $X$ be a complex Stein space. There is a natural bijection
  \begin{equation*}
   [X, \mathrm{Gr}_{n/ \mathbb{C} } ] \overset{\sim}{\longrightarrow} \{ \text{holomorphic vector bundles of rank $n$ over $X$} \} / \cong 
  \end{equation*}
  between homotopy classes of maps from $X$ into the infinite Grassmannian $\mathrm{Gr}_{n}$ over $\mathbb{C}$
  and isomorphism classes of holomorphic vector bundles of rank $n$ over $X$.
 \end{IThm}
 Grauert's theorem is stated in terms of classical homotopy classes and with $[0,1]$ as an interval.
 L\'arusson showed that Grauert's theorem can also be phrased in terms of
 motivic homotopy theory with $\mathbb{C}$ as an interval \cite{Larusson03,Larusson04,Larusson05}.
  
 Morel's theorem and Grauert's theorem seem to be linked
 via analytification. But as Morel's theorem is valid without restrictions on the base field, it seems natural to ask
 if the same holds for other analytic fields, in particular for a complete non-archimedean valued field $k$ such as
 a field of $p$-adic numbers $\mathbb{Q}_p$ or a field $\kappa ((T))$ of formal Laurent series.
  
 The starting point for this project was Ayoub's $\mathbb{B}^1$-homotopy theory for rigid analytic varieties.
 It is a version of Morel--Voevodsky's homotopy theory of a site with interval where
 the ``unit ball'' $\mathbb{B}^1 := \operatorname{Sp} k \langle T \rangle$ figures as an interval.
 There are several approaches to non-archimedean analytic geometry. 
 Motivated by Ayoub's work, we chose to work with rigid analytic varieties, too.
 We could as well have considered Berkovich spaces or adic spaces as the three approaches are equivalent in the
 setting we are interested in.
 Ayoub's approach was well suited to define and work with motives of rigid analytic varieties ~\cite{AyoubB1,Vezzani14b}.
 
 However, it turns out that for our purposes the analytification $\mathbb{A}^1_\mathrm{rig}$ of the affine line 
 is a better choice of interval. 
 Theorem~\ref{ThmClassIntro} becomes false if one replaces the interval object $\mathbb{A}^1_\mathrm{rig}$ by $\mathbb{B}^1$:
 Let $k = \mathbb{Q}_p$, $p \geq 3$ and
 \begin{equation*}
  X = \operatorname{Sp} \big( \mathbb{Q}_p \langle T_1 , T_2 \rangle / (T_1 ^2 - T_2 (T_2 - p )(T_2 - 2 p )) \big).
 \end{equation*}
 The corresponding assignment 
 \begin{equation*}
  [X, \mathbb{P}^\infty ]_{\mathbb{B}^1_\mathrm{rig}} \longrightarrow
  \{ \text{analytic $k$-line bundles over $X$} \} / \cong 
 \end{equation*}
 is not even well-defined since on the right hand side there exist non-isomorphic line bundles on $X$ 
 which would correspond to the same homotopy class on the left hand side. 
 This is worked out in Example \ref{Exbadreduction}.
 
 The questions about an $\mathbb{A}^1_\mathrm{rig}$-classification of vector bundles
 and about a $\mathbb{B}^1$-classification of vector bundles are quite different and 
 have to be answered with different methods.
 The method we chose seems to be good for the interval object $\mathbb{A}^1_\mathrm{rig}$ but less so for the interval object $\mathbb{B}^1$.
 
\subsection*{The proof}
 It turns out that homotopy invariance with respect to the interval is really the
 key question for a homotopy classification of vector bundles over (suitable) rigid analytic varieties.
 For an interval object $I \in \{ \mathbb{B}^1 , \mathbb{A}^1_\mathrm{rig} \}$ 
 we say that ``vector bundles of rank $n$ are $I$-invariant'' on a subcategory $\mathcal{R}$ of the category of rigid analytic varieties 
 if for each object $X \in \mathcal{R}$, the projection $X \times I \rightarrow X$
 induces a bijection on isomorphism classes of vector bundles of rank $n$.
 
 Let us give an outline of the proof of Theorem~\ref{ThmClassIntro}. Let $k$ be a complete, non-archimedean,
 non-trivially valued field and $X$ a smooth rigid analytic variety.
 \begin{enumerate}
  \item Every line bundle over $X \times \mathbb{A}^ 1_\mathrm{rig}$ has a local trivialisation
      by subsets of the form $\{ U_i \times \mathbb{A}^1_\mathrm{rig} \} _{i \in J}$ 
      where $\{ U_i \} _{i \in J}$ is an admissible covering of $X$.
      This is Corollary~\ref{A1local}, a corollary of Theorem~\ref{ThmTamme} 
      by Kerz, Saito and Tamme~\cite{KerzSaitoTamme16}.
  \item Consider a line bundle on $X \times \mathbb{A}^ 1_\mathrm{rig}$ and a local trivialisation of the form
      \mbox{$\{ U_i \times \mathbb{A}^1_\mathrm{rig} \} _{i \in J}$}. 
      Then every transition function  
      on $U_i \cap U_j \times \mathbb{A}^1_\mathrm{rig}$ 
      is constant in the $\mathbb{A}^ 1_\mathrm{rig}$-direction.
      This is Theorem~\ref{cocycA1}.
  \item The projection $X \times \mathbb{A}^1_\mathrm{rig} \rightarrow X$ induces an isomorphism of groups
      $\operatorname{Pic} (X) \cong \operatorname{Pic} (X \times \mathbb{A}^1_\mathrm{rig} )$, i.\,e., line bundles are 
      $\mathbb{A}^1_\mathrm{rig}$-invariant on smooth rigid analytic varieties (Theorem~\ref{rigA1inv}).
  \item Every vector bundle over a quasi-Stein rigid analytic variety admits a finite local trivialisation (Theorem~\ref{SteinfinG}).
  \item The assumptions of Theorem~\ref{ThmAHWIntro} are satisfied if $n = 1$, $I = \mathbb{A}^ 1_\mathrm{rig}$ and
      $\mathcal{R}$ is the subcategory of quasi-Stein spaces. 
      Theorem~\ref{ThmClassIntro} follows.
 \end{enumerate}

 We expect Theorem~\ref{ThmClassIntro} to generalise to vector bundles of higher rank. 
 That idea was proposed by Matthias Wendt.
 \begin{IConj}\label{Conjrank}
  Let $X$ be a smooth rigid analytic quasi-Stein variety over a complete, non-archimedean, non-trivially valued field $k$. 
  Then there is a natural bijection
  \begin{equation*}
   [X, \mathrm{Gr}_{n/k} ]_{\mathbb{A}^1_\mathrm{rig}} \overset{\sim}{\longrightarrow} 
   \{ \text{analytic $k$-vector bundles of rank $n$ over $X$} \} / \cong 
  \end{equation*}
  between motivic homotopy classes from $X$ into the infinite Grassmannian $\mathrm{Gr}_{n/k}$
  and isomorphism classes of analytic $k$-vector bundles of rank $n$ over $X$.
 \end{IConj}
  
 Such a statement would be desirable because one expects geometric applications.
 Both Morel's theorem and Grauert's theorem set the stage for other theorems that are interesting on their own.
 Using Morel's theorem and the theory of algebraic Euler classes also constructed by Morel \cite[§8]{Morel12}, 
 Asok and Fasel proved splitting theorems for algebraic vector bundles \cite{AsokFasel14a,AsokFasel14b,AsokFasel14c,AsokFasel15}.
 Grauert's theorem is strongly linked to Gromov's h-principle, a powerful tool in differential topology.
 There should be geometric applications of Theorem~\ref{ThmClassIntro}, the conjecture and, more generally,
 non-archimedean motivic homotopy theory along the same lines.

 For a proof it remains to show that vector bundles over smooth quasi-Stein varieties are
 \mbox{$\mathbb{A}^1_\mathrm{rig}$-}\nobreak\hspace{0pt}invariant. This means, if $X$ is smooth and quasi-Stein, 
 then every vector bundle over $X \times \mathbb{A}^ 1_\mathrm{rig}$ is isomorphic to the pullback along 
 $X \times \mathbb{A}^1_\mathrm{rig} \to X$ of some vector bundle over $X$.
 Then the assumptions of Theorem~\ref{ThmAHWIntro} would be satisfied and the conjecture would follow.
    
\subsection*{Structure of the article}
 The second section treats rigid analytic varieties and their vector bundles. In particular, we prove Serre--Swan theorems
 for rigid analytic varieties (Theorems~\ref{rigSerreSwan} and~\ref{ZarSerreSwan}).
  
 The third section treats the question of homotopy invariance with respect to an interval object~$I$.
 We give examples where \mbox{$I$-}\nobreak\hspace{0pt}invariance is violated as well as the positive results,
 in particular Theorem~\ref{rigA1inv}: Line bundles over a smooth $k$-rigid analytic variety are 
 \mbox{$\mathbb{A}^1_\mathrm{rig}$-}\nobreak\hspace{0pt}invariant.
 
 In the fourth section we define motivic homotopy theories for rigid analytic varieties
 in the spirit of Morel--Voevodsky's \mbox{$\mathbb{A}^1$-}\nobreak\hspace{0pt}homotopy theory and Ayoub's 
 \mbox{$\mathbb{B}^1$-}\nobreak\hspace{0pt}homotopy theory (Proposition~\ref{rudel}).
 We show that $\mathbb{A}^1_\mathrm{rig}$ is contractible in the 
 \mbox{$\mathbb{B}^1$-}\nobreak\hspace{0pt}homotopy category (Lemma~\ref{Manta}).
 
 In the fifth section we prove Theorem~\ref{ThmAHWIntro}/Theorem~\ref{ThmClass}, following the proof
 of the corresponding theorem in $\mathbb{A}^1$-homotopy theory by Asok--Hoyois--Wendt.
 Now we can deduce Theorem~\ref{ThmClassIntro}/Theorem~\ref{ThmClassLines}.
 
\subsection*{Relation to PhD thesis}
  This is an updated and abridged version of the author's PhD thesis~\cite{Sigloch16} at the University of Freiburg.
  The thesis is publicly available online at \url{https://freidok.uni-freiburg.de/data/11742}.
  You may want to read parts of the thesis instead of this article if you wish for more
  background about rigid analytic varieties or about model categories.
  You might also want to take a look if you are interested in the divisor class group.
  You can find a survey of known results about the divisor class group, 
  its relation to the Picard group and homotopy invariance there. 
  All new results of the thesis are contained in this article.
 
 \subsection*{Acknowledgements}
  First of all, I want to thank Matthias Wendt for advise and patience, for giving me this intriguing question 
  and making me learn a lot of very beautiful mathematics.
  I thank Annette Huber-Klawitter for approving of this project and this way making the whole project possible
  in the first place.
  I thank
  Peter Arndt, Aravind Asok, Joseph Ayoub, Federico Bambozzi, Antoine Ducros, Carlo Gasbarri, Fritz H{\"o}rmann,
  Annette Huber-Klawitter, Marc Levine, Werner L\"{u}tkebohmert, Florent Martin, Vytautas Pa\v{s}k\={u}nas,
  J\'{e}r{\^o}me Poineau, Dorin Popescu, Shuji Saito, Marco Schlichting, 
  Georg Tamme, Konrad V{\"o}lkel and Matthias Wendt for discussions and  
  J\'{e}r\^{o}me Poineau for pointing out a very stupid mistake.
  I thank Maximilian Schmidtke and Eva Nolden for reading a draft of parts of my PhD thesis and pointing out typos and bad language.
  Once again, I thank Matthias Wendt for reading several versions of the thesis and the article carefully and for his feedback
  concerning both mathematics, style and the exposition.
  I thank Jens, Konrad, Matthias, Oliver and Shane for improving the introduction.
  I am very grateful to Joseph Ayoub for acting as referee for this thesis, for reading it thoroughly and for his detailed comments.

\tableofcontents

\section{Rigid analytic varieties and their vector bundles}
 \subsection{Preliminaries and conventions}
 Throughout this article we use the following conventions:
 All rings are commutative with 1.
 All algebras are associative with 1.
 All valuations are non-trivial.
 All rigid analytic varieties are separated.

 Let $k$ always be a complete, non-archimedean, non-trivially valued field.
 For the theory of rigid analytic
 varieties we refer to Bosch--G\"untzer--Remmert \cite{BGR} and Fresnel--van der Put \cite{FresnelPut}.
 The reader who is not familiar with rigid analytic varieties might prefer to read chapter~1 of 
 \href{https://freidok.uni-freiburg.de/data/11742}{the author's thesis} \cite{Sigloch16} instead of this section.
 \begin{Notation}
 We denote the Tate algebra by
  \begin{align*}
   k \langle T_1 , \dotsc , T_n \rangle &:= \left\{ \sum _{i= ({i_1}, \dotsc , {i_n}) \in \N ^n} a_i T_1^{i_1} \dotsm T_n^{i_n} 
      \middle| a_i \in k \text{ for all $i$ and } \vert a_i \vert \xrightarrow{i_1+ \dotsb + i_n \to \infty} 0 \right\}\\
   \text{and set} \quad \B ^1 &:= \operatorname{Sp} (k \langle T \rangle ).
  \end{align*}
 \end{Notation}
 
 Recall that affinoid varieties have a canonical reduction:
 Let $A$ be an affinoid $k$-algebra. The \emph{reduction} of an affinoid $k$-algebra $A$ is
 $\tilde{A} :=  A^\circ / A^{\circ\circ}$
 where $A^\circ \subset A$ is the subring of power-bounded elements and $A^ {\circ \circ}$ the set of topologically nilpotent
 elements of $A$. The set $A^{\circ \circ}$ is an ideal in $A^\circ$. 
 Reduction defines a functor from the category of affinoid $k$-algebras to the category of
 affine $\tilde{k}$-algebras \cite[6.3]{BGR}.

 \begin{Def}
  The formal scheme $\mathcal{X} := \operatorname{Spf} (A^\circ )$ is called the 
  \emph{canonical model} of $X = \Sp A$.
  The special fibre of $\mathcal{X}$ is denoted by $\tilde{X}^c$.
  It is called the \emph{canonical reduction} of $X$.
 \end{Def}
 
 \subsection{Quasi-Stein spaces and their algebras}
 Complex Stein spaces are complex spaces which have many global holomorphic functions.
 Therefore they have the good properties needed for Grauert's classification of holomorphic vector bundles.
 For example, they satisfy Cartan's theorems A and B. 
 For further reading about complex Stein spaces, we refer to Grauert--Remmert~\cite{GrauertRemmert04} and 
 Forstneri{\v{c}}~\cite{Forstneric11}.
 
 One non-archimedean analogue of complex Stein spaces are Kiehl's quasi-Stein spaces:
 \begin{Def}[quasi-Stein {\cite[Definition~2.3]{Kiehl67}}]\label{DefStein}
  A rigid analytic space $X$ is \emph{quasi-Stein}\index{rigid analytic variety!quasi-Stein}\index{quasi-Stein}
  if there is an admissible covering by open affinoid subspaces 
  \begin{equation*}
   U_1 \subset U_2 \subset U_3 \subset \dotsb
  \end{equation*}
  such that for all $i$, the image of $\mathcal{O}_X (U_{i+1})$ is dense in $\mathcal{O}_X (U_i)$.
 \end{Def}
 \begin{Ex}[The affine $n$-space $\mathbb{A}^n_\text{rig}$ {\cite[9.3.4 Example 1]{BGR}}]\label{DefA1rig}\index{$\mathbb{A}^n_\text{rig}$}
  Let $\eta \in k$ with $\vert \eta \vert > 1$ and $i, n \in \N$. Set
  \begin{equation*}
   A_i := k \langle \eta ^{-i} T_1, \dotsc , \eta ^ {-i} T_n \rangle .
  \end{equation*}
  This means that $\operatorname{Sp} (A_i) = \mathbb{B}^ n_{\vert \eta \vert ^i}$ is the $n$-dimensional ball of polyradius 
  $\vert \eta \vert ^i$.
  The algebras $A_i$ form a chain
  \begin{equation*}
   k \langle T_1 , \dotsc , T_n \rangle = A_0 \supsetneq A_1 \supsetneq A_2 \supsetneq \dotsb \supsetneq k [T_1 , \dotsc , T_n ]
  \end{equation*}
  and the inclusions $A_i \supset A_{i+1}$ correspond to inclusions of rational subsets
  \begin{equation*}
   \mathbb{B}^n _{\vert \eta \vert ^i} \subset \mathbb{B}^n _{\vert \eta \vert ^{i+1}}.
  \end{equation*}
  The \emph{affine $n$-space} $\mathbb{A}^n_\text{rig}$ is defined as the colimit
  \begin{equation*}
   \mathbb{A}^n_\text{rig} := \bigcup _{i \in \mathbb{N}} \mathbb{B}^n _{\vert \eta \vert ^i}
  \end{equation*}
  of all those inclusions of rational subsets. It does not depend on the choice of the element $\eta$.
  The affine spaces $\mathbb{A}^n_\text{rig}$ are quasi-Stein.
 \end{Ex}
  
 \begin{Thm}[Kiehl {\cite[Satz~2.4]{Kiehl67}}]\label{ThmAB}
  Let $X$ be a quasi-Stein space and $\{ U_i \} _{i \in \mathbb{N}}$ an admissible covering as in Definition 
  \ref{DefStein}. Let $\mathcal{G}$ be a coherent sheaf on $X$.
   Then the following hold:
  \begin{enumerate}
   \item The image of $\mathcal{G} (X)$ is dense in $\mathcal{G} (U_i)$ for all $i$.
   \item The cohomology groups $H^ i (X, \mathcal{G} )$ vanish for $i > 0$ (Theorem~B).
   \item\label{ThmA} For each $x \in X$, the image of $\mathcal{G} (X)$ in the stalk $\mathcal{G}_x$ 
      generates this stalk as an $\mathcal{O}_{X,x}$-module (Theorem~A).
  \end{enumerate}
 \end{Thm}
 
 \begin{Def}[Quasi-Stein algebra]
  Let $X$ be a rigid analytic quasi-Stein variety. The $k$-algebra $\mathcal{O}_X (X)$ of global functions on $X$
  is called a \emph{quasi-Stein algebra}\index{quasi-Stein algebra}.
 \end{Def}
 Quasi-Stein algebras are Fr{\'e}chet algebras, but in general not Banach. They are complete, but their topology does not necessarily
 arise from a norm. As the following example shows, quasi-Stein algebras are in general not noetherian. 
 \begin{Ex}[compare {\cite[Remark 3 on p.~179]{GrauertRemmert04}}]
  Let $X$ be a quasi-Stein space that contains an infinite discrete subset $D$ which does not have a limit point in $X$. 
  Let $A = \mathcal{O}_X (X)$ be the corresponding quasi-Stein algebra.
  The non-archimedean analogue of Weierstrass' product theorem holds: Given an infinite discrete subset $D' \subset X$ which
  does not have a limit point in $X$ and an assignment
  \begin{align*}
   D' &\longrightarrow \mathbb{N}\\
   d &\longmapsto m(d)
  \end{align*}
  there exists a function $f \in \mathcal{O}_X (X)$ such that for all $d \in D'$, the function $f$ has a zero in $d$ of multiplicity $m(d)$.
  The proof is the same as in the complex case, e.\,g.~as in Rudin's book~\cite{Rudin87}. 
  The elements of the ideal 
  \begin{equation*}
   I = \{ f \in A \mid f(x) = 0 \text{ for almost all } x \in D \}
  \end{equation*}
  have no common zeroes. So a maximal ideal $\mathfrak{m}$ containing $I$ cannot be finitely generated.
 \end{Ex} 
 Quasi-Stein algebras may not be noetherian but they still behave very much like affinoid algebras:
 \begin{Thm}[Bambozzi--Ben Bassat--Kremnizer]\label{ThmFederico}
  The category of $k$-rigid analytic quasi-Stein varieties and the category of quasi-Stein $k$-algebras are contravariantly equivalent.
 \end{Thm}
 This was proven by Federico Bambozzi (private communication).
 The same proof as for \cite[Theorem~4.25]{BambozziBenBassatKremnizer15} works.
 
 \begin{Rem}\label{Remcateq}  
  In this respect, quasi-Stein algebras are analogous to affine algebras in algebraic geometry and Stein algebras in complex analysis.
  In complex analysis, Stein algebras were introduced by Grauert \cite[§2, Definition 1]{Grauert63}.
  They were intensively studied by Forster \cite{Forster64, Forster67}. Forster proved in particular that 
  the category of complex Stein spaces and the category of complex Stein algebras are contravariantly equivalent \cite[§1, Satz 1]{Forster67}.
  The corresponding statement for affine algebraic varieties (or even schemes) is a tautology.
 \end{Rem}
  
 \begin{Rem}
  \begin{enumerate}
   \item\label{RemSteina} Whereas the notion of a non-archimedean quasi-Stein space is unambiguous, 
      there are several different types of non-archimedean spaces
      which bear the name \emph{Stein space}: 
      The original and most restrictive notion was defined by Kiehl~\cite[Definition~2.3]{Kiehl67}.
      We call these spaces \emph{Kiehl--Stein spaces}.
      Kiehl--Stein spaces were intensively studied by L\"utkebohmert~\cite{Lutke73}.
      Kiehl--Stein spaces are in particular quasi-Stein.
      
      A weaker notion of Steinness was introduced by Liu \cite{Liu88,Liu89,Liu90}. 
      Liu calls a rigid analytic space $X$ Stein if for every coherent sheaf $\mathcal{F}$ on $X$ all higher cohomology
      groups $H^i(X, \mathcal{F}), i \geq 1,$ vanish. In order to distinguish these spaces from Kiehl--Stein spaces
      we call them \emph{cohomologically Stein spaces}. Liu constructed compact cohomological Stein spaces which are not affinoid.
      But every compact quasi-Stein space is affinoid by definition,
      hence Liu's spaces are not quasi-Stein. Liu moreover proved the analogue of Theorem~\ref{ThmFederico} 
      for compact cohomological Stein spaces and their algebras
      \cite[Propositions 1.3 and 3.2]{Liu90}.
   \item Correspondingly, there are different notions of Stein algebras in non-archimedean geometry.
      Kiehl--Stein algebras appear implicitly in L\"utkebohmert's 1973 article \cite{Lutke73}.
      Liu--Stein algebras, i.\,e., the algebras of global functions on compact cohomologically Stein spaces, appear in Liu's work.
      We already discussed this in \ref{RemSteina}).
      More general Fr{\'e}chet--Stein algebras and their coadmissible modules 
      were introduced by Schneider and Teitelbaum \cite{SchneiderTeitelbaum03} 
      in the context of $p$-adic analytic groups and Langlands theory. 
   \item There are approaches to unify the archimedean and the non-archimedean theory:
      Poineau constructed Berkovich spaces over $\mathbb{Z}$, 
      capturing both the archimedean and the non-archimedean theory \cite{Poineau13}.
      He works with cohomologically Stein spaces.
      Bambozzi--Ben Bassat--Kremnizer investigate the topology of Kiehl--Stein spaces 
      over any valued base field \cite{BambozziBenBassatKremnizer15}.
      They work with dagger spaces to unify the approaches.
  \end{enumerate}
 \end{Rem}
 
\subsection{Flat morphisms}
 A morphism of affinoid algebras is called \emph{flat} if it is flat as a morphism of rings.
 \begin{Lemma}\label{flats}
  Let $A$ be an affinoid algebra.
  Then $A \rightarrow A \langle T \rangle$ is flat.
  If $\vert f \vert \leq 1$, then $A \rightarrow A \langle f^{-1} \rangle$ is flat.
  Inclusion of an affinoid subdomain is flat.
 \end{Lemma}
 \begin{proof}
  The morphisms $A \rightarrow A[T]$ and $A \rightarrow A_f$ are flat. 
  Completion is faithfully flat \cite[§3, no.~4 Th\'eor\`eme~3 and no.~5 Proposition~9]{BourbakiAlgComm3}.
  Inclusion of an affinoid subdomain is flat \cite[4.1 Corollary~5]{Bosch14}.
 \end{proof}
 \begin{Lemma}\label{flatStein}
  Let
  \begin{equation*}
   X = \operatorname{colim} ( U_1 \subset U_2 \subset U_3 \subset \dotsb )
  \end{equation*}
  be a quasi-Stein variety with 
  $U_i = \Sp A_i$.
  Then $A := \mathcal{O}_X(X) \rightarrow A_i$ is flat for all $i$.
 \end{Lemma}
 \begin{proof}
  Using Kiehl's Theorem~B \cite[Satz 2.4]{Kiehl67}, the proof works just as Gruson's proof of \cite[V, Corollaire 1]{Gruson68}:
  Let $I \subset A$ be any finitely generated ideal. We want to show that the morphism
  \begin{equation}\label{flatthmbeq}
   I \otimes _A A_i \longrightarrow A \otimes _A A_i
  \end{equation}
  is injective. Then $A_i$ is flat over $A$ by \cite[tag~00HD]{stacks-project}.
  
  The ideal $I$ gives rise to a coherent sheaf $\mathcal{I}$ on $X$. For each $j \geq i$,
  $A_j$ is flat over $A_i$ by Lemma~\ref{flats}. Therefore, for each $j \geq i$,
  the morphism
  \begin{equation*}
   \mathcal{I} (U_j) \otimes _{A_j} A_i \longrightarrow A_j \otimes _{A_j} A_i
  \end{equation*}
  is injective. 
  By Kiehl's Theorem~B \cite[Satz 2.4]{Kiehl67}, the derived limit $\operatorname{lim} ^ 1 _j  (\mathcal{I} (U_j))$ vanishes
  and hence the morphism \eqref{flatthmbeq} is injective, too.
 \end{proof}
 
 We want to use Lemma~\ref{flatStein} in combination with the following lemma. It is the Banach version of a well-known theorem
 in commutative algebra.
 \begin{Lemma}\label{Lam}
  Let $A \overset{\phi}{\longrightarrow} A'$ be a flat morphism of commutative Banach algebras. Let $M$ be a
  finitely presented complete normed $A$-module and $N$ a finitely generated complete normed $A$-module. 
  Then the natural map
  \begin{equation*}  
   \hat{\sigma} \colon A' \hat{\otimes} _A \Hom _A (M,N) \longrightarrow Hom _{A'} (A' \hat{\otimes} _A M,
   A' \hat{\otimes} _A N)
  \end{equation*}
  is an isomorphism.
 \end{Lemma}
  To prove this, we want to use the fact that there is no essential difference between the category of
  finitely generated complete normed $A$-modules with continuous $A$-linear maps as morphisms and
  the category of finitely generated $A$-modules with $A$-linear maps as morphisms without demanding continuity \cite[§3.7.3]{BGR}.
  During the proof, Hom-sets in the first named category are denoted by $\Hom ^{\textit{cont}}$,
  morphisms being both linear and continuous. On the other hand, Hom-sets in the second named category 
  consist of morphisms that only have to be linear and will be denoted by $\Hom ^{\textit{lin}}$.
 \begin{proof}
  We only need to prove that $\hat{\sigma}$ is bijective.
  There is a commutative diagram
  \begin{equation*}
   \begin{xy} \xymatrix{
    A' \hat{\otimes} _A \Hom ^{\textit{cont}} _A (M,N) \ar[r]^{\hat{\sigma}} & 
       \Hom _{A'}^{\textit{cont}} (A' \hat{\otimes}_A M , A' \hat{\otimes}_A N) \\
    A' \otimes _A \Hom _A^{\textit{lin}} (M,N) \ar[u]^\wr _\alpha \ar[r]^\sim _\sigma &
       \Hom _{A'} ^{\textit{lin}} (A' \otimes _A M, A' \otimes _A N) \ar[u]^\wr _\beta
   } \end{xy} .
  \end{equation*}
  The maps $\alpha$ and $\beta$ are bijections by \cite[§3.7.3 Proposition~2, Proposition~6]{BGR}. The map $\sigma$ is 
  an isomorphism by \cite[Proposition~I.2.13]{Lam06} or \cite[Proposition~2.10]{Eisenbud}. Thus, $\hat{\sigma}$ is bijective, too.
 \end{proof}
 
  \begin{Def}[{\cite{AyoubB1}}]
  \begin{enumerate}
    \item An affinoid $k$-variety $X= \operatorname{Sp} A$ is called \emph{regular in $x \in X$} if 
      $\mathcal{O}_{X,x}$ is regular. Furthermore, $X$ is called \emph{smooth in $x$} if for every finite field extension
      $k \subset k'$, the affinoid $k'$-variety $\operatorname{Sp} (A \hat{\otimes} _k k')$ is regular in each point $x'$ lying above $x$.
   \item A rigid analytic variety $X$ is called \emph{regular} (\emph{smooth}) if it is regular (smooth) in all $x \in X$.
   \item A morphism $f\colon X \rightarrow Y$ of rigid $k$-varieties is \emph{smooth} if it is flat and for each  $y \in Y$ 
      the morphism restricted to the fibre $f^{-1}(y) \rightarrow \kappa (y)$ is smooth. 
  \end{enumerate}
 \end{Def}
 
 \begin{Def}[$\rigVar$]\index{$\rigVar$}
  The category of smooth rigid analytic $k$-varieties with analytic morphisms is denoted by $\rigVar _k$. 
  The subscript $k$ will be suppressed in the notation except when there is a chance of confusion.
 \end{Def}

 \subsection{Vector bundles}
  Let $X$ be a locally G-ringed space with structure sheaf $\mathcal{O}_X$ of $k$-algebras.
  \begin{Def}
   A \emph{vector bundle of rank $n$} is a locally free $\mathcal{O}_X$-module of rank $n \in \N$.
   That means that it is a sheaf $\mathcal{F}$ of $k$-algebras such that 
   \begin{enumerate}
    \item for each admissible subset $U \subset X$ the ring of sections $\mathcal{F} (U)$ 
        is an $\mathcal{O}_X(U)$-module and 
    \item there exists an admissible covering $\{ U_i \} _{i \in I}$ of $X$ 
        such that all $\mathcal{F} (U_i)$ are free of rank $n$.
   \end{enumerate}
   Such a covering is called a \emph{local trivialisation} of $\mathcal{F}$. 
  \end{Def}
 
  When we speak of a \emph{vector bundle on $X$}, we make implicit choices: 
  The choice of G-topology and the choice of structure sheaf. 
  We will be careful to always specify which G-topology we are referring to.
  On the contrary, there will be no chance of confusion about the structure sheaf.
  
  Every rigid analytic variety is endowed with a G-topology. 
  Affinoid varieties additionally carry the Zariski topology. 
  It is coarser than the G-topology and therefore gives rise to an (a priori) different notion of vector bundle.
  \begin{Deff}[analytic/Zariski vector bundle]
   \begin{enumerate}
    \item Vector bundles with respect to the G-topology of a rigid analytic variety $X$ are called \emph{analytic vector bundles}.
      The category of analytic vector bundles on $X$ is denoted by $\operatorname{VB} (X)$. 
    \item If $X$ is affinoid, vector bundles with respect to the Zariski topology on $X$ are called \emph{Zariski vector bundles}.
     The category of Zariski vector bundles on $X$ is denoted by $\operatorname{VB}_\mathrm{Zar} (X)$.
   \end{enumerate}
  \end{Deff}
  In fact the two notions are the same as we will see in Corollary~\ref{vbequiv}.
 
  \begin{Rem}
   The definition of an analytic vector bundle coincides with the usual definition of a vector bundle on a rigid analytic variety
   {\cite[p.~87]{FresnelPut}}. 
  \end{Rem}
  
 \begin{Def}
  The functor $\vect ^n$ assigns to a rigid analytic variety $X$ the set of isomorphism classes of vector bundles 
  of rank $n$ on $X$, with the trivial vector bundle of rank $n$ as a base point.
  To a morphism $f \colon X \to Y$ of rigid analytic varieties it assigns
  the pullback map $f^\ast \colon \vect ^n (Y) \to \vect ^n (X)$.     
 \end{Def}
 Later on, we need a functor of vector bundles of rank $n$, not just a functor of \emph{isomorphism classes} of vector bundles.
 For each rigid analytic variety $X$ let
 \begin{equation*}
  \operatorname{VB} _n \colon X \longmapsto \{ \mathcal{E} \text{ vector bundle of rank $n$ over } X \} .
 \end{equation*}
 This $\operatorname{VB} _n$ is not a functor.
 Let $\mathcal{E} \in \operatorname{VB} _n(X)$ be a vector bundle and
 \begin{equation*}
  X \overset{f}{\longrightarrow} Y \overset{g}{\longrightarrow} Z
 \end{equation*}
 be morphisms of rigid analytic varieties. The vector bundles $f^\ast g^\ast \mathcal{E}$ and $(gf)^\ast \mathcal{E}$
 are isomorphic but in general not equal.
 A way to work around this is to replace the pseudofunctor $\operatorname{VB} _n$ by Grayson's functor 
 $\Vect ^n$ of big vector bundles.
 \begin{Def}[big vector bundle, big principal bundle, cf.~{\cite[C.4]{FriedlanderSuslin02}}, {\cite{Grayson95}}]\label{Defstrictvect}
  Let $\mathcal{R}$ be a subcategory of a small category equivalent to $\rigVar$.
  In short, a \emph{big vector bundle}\index{big vector bundle} on a rigid analytic variety $X \in \operatorname{Ob} ( \mathcal{R})$
  is a family of vector bundles together with isomorphism data.
  The construction is as follows. The overcategory $\mathcal{R} / X$ has as objects morphisms 
  $(f\colon Y \to X ) \in \operatorname{Mor} ( \mathcal{R} )$. The morphisms of $\mathcal{R} / X$ are commutative triangles.
  Now a big vector bundle over $X$ is a family of vector bundles 
  \begin{equation*}
   \{ \mathcal{E}_Y \}_{(f\colon Y \to X) \in \operatorname{Ob} (\mathcal{R} /X)}
  \end{equation*}
  together with compatible isomorphisms
  \begin{equation*}
   \{ \varphi \colon g^\ast \mathcal{E}_Z \to \mathcal{E} _Y \mid (g\colon Y \to Z) \in \operatorname{Mor} ( \mathcal{R} / X) \} .
  \end{equation*}
  The assignment
  \begin{equation*}
   \Vect ^n \colon X \longmapsto \{ \text{big vector bundles of rank $n$ over } X \}
  \end{equation*}
  defines a functor from $\mathcal{R}$ to groupoids. Defining big principal bundles analogously, also
  \begin{equation*}
   \mathcal{P} ^n \colon X \longmapsto \{ \text{big $\operatorname{GL}_n$-principal bundles over } X \}
  \end{equation*}
  defines a functor from $\mathcal{R}$ to groupoids.
 \end{Def} 
 
 \begin{Rem}
  It is quite obvious that for each $X \in \operatorname{Ob} ( \mathcal{R} )$ the category of big vector bundles over $X$ and
  the category of vector bundles over $X$ are equivalent.
 \end{Rem}
 
  \begin{Def}[Line bundle]
   Let $X$ be a rigid analytic variety. A locally free $\mathcal{O}_X$-module of rank one is called a \emph{line bundle}\index{line bundle}.
  \end{Def}
  
  As usual, line bundles form a group under $\otimes$, the \emph{Picard group}.
  
  The Picard group of an affinoid variety is strongly linked to the variety's canonical formal model and canonical reduction:
  \begin{Thm}[Gerritzen {\cite[§4]{Gerritzen77}}]\label{ThmGerritzenPicc}
   Let $k$ be a complete discretely, non-trivially valued field. Let $A$ be a reduced $k$-affinoid algebra 
   whose ring of integers $A^\circ$ is regular. Assume $\Vert A \Vert = \vert k \vert$. Then there is an isomorphism
   \begin{equation*}
   \operatorname{Pic} (A) \cong \operatorname{Pic} ( \tilde{A} ).
   \end{equation*}
  \end{Thm}  
  
 \subsection{Berkovich spaces}
  In the proof of the Serre--Swan theorem for rigid analytic quasi-Stein spaces we are going to work with
  the associated Berkovich spaces. Therefore we need the following theorem by Berkovich.
  \begin{Thm}[Berkovich {\cite[Theorem~1.6.1]{Berkovich93}}]\label{Berkequivalence}
   The category of paracompact strictly $k$-analytic Berkovich spaces and the category of quasiseparated rigid analytic $k$-varieties
   that have an admissible covering of finite type are equivalent. An affinoid covering is said to be \emph{of finite type} 
   if each of its members intersects only finitely many of the other members.
  \end{Thm}
  Berkovich spaces are topological spaces and thus come with a notion of continuous vector bundle.
  But we want to take the analytic structure into account.
  \begin{Def}[Vector bundle on a Berkovich space]
   Let $X$ be a Berkovich space.
   Let $\mathcal{O}_X$ be the structure sheaf of analytic functions on $X$.
   An \emph{analytic vector bundle} on $X$ is a locally free $\mathcal{O}_X$-module.
  \end{Def}

  \begin{Prop}[Berkovich]\label{vbequivBer}
   Let $X$ be a quasiseparated rigid analytic space and $X_B$ the corresponding Berkovich space.
   If every point $x \in X_B$ has an affinoid neighbourhood, then the category of analytic vector bundles on $X$ and 
   the category of analytic vector bundles on $X_B$ are equivalent.
  \end{Prop}
  \begin{proof}
   Berkovich calls an analytic space $Y$ \emph{good}\index{Berkovich space!good} if every point $x \in X_B$ 
   has an affinoid neighbourhood~\cite[Remark 1.2.16]{Berkovich93}. 
   If $Y$ is good, the category of analytic vector bundles with respect to the Berkovich topology on $Y$
   and the category of vector bundles with respect to the G-topology on~$Y$ are equivalent~\cite[Proposition~1.3.4]{Berkovich93}.
   The category of vector bundles with respect to the G-topology on~$Y$ and the category of analytic vector bundles on the
   associated rigid analytic space are equivalent~\cite[bottom of p.~37]{Berkovich93}.
   Berkovich's G-topology on $Y$ is defined on \cite[p.~25]{Berkovich93}. The G-topology on the corresponding
   rigid analytic space is the strong G-topology (for a definition see \cite[pp.~80f]{FresnelPut}, \cite[9.1.4]{BGR}).
  \end{proof} 
   
 \subsection{The Serre--Swan theorem}
   We now prove the rigid analytic version of the Serre--Swan theorem: Isomorphism classes of vector bundles of rank $n$
   over a rigid analytic quasi-Stein variety~$X$ correspond to isomorphism classes of projective modules of rank $n$ over the ring
   of global functions~$\mathcal{O}_X (X)$. There is a general principle for proving Serre--Swan theorems.
   It was analysed by A.~Morye~\cite{Morye} and we give a short account of her argument.
   \begin{Def}
    Let $(X, \mathcal{T}, \mathcal{O}_X)$ be a locally G-ringed space.
    A coherent sheaf of $\mathcal{O}_X$-modules is called \emph{finitely generated by global sections} if
    there is a finite set of global sections whose images in the stalk at any point $x \in X$
    generate that stalk as an $\mathcal{O}_{X,x}$-module.
   \end{Def} 
   
   \begin{Def}[admissible subcategory {\cite[Definition 1.1]{Morye}}]
    Let $(X, \mathcal{O}_X)$ be a locally G-ringed space. A subcategory $\mathcal{C}$ of the category of $\mathcal{O}_X$-modules
    is called an \emph{admissible subcategory} if
    \begin{itemize}
     \item[{\bf C1}] $\mathcal{C}$ is a full abelian subcategory of the category of $\mathcal{O}_X$-modules with internal $\Hom$, 
         i.\,e., for $\mathcal{F}, \mathcal{G} \in \operatorname{Ob} ( \mathcal{C} )$ the sheaf
         $\operatorname{Hom} _{\mathcal{O}_X} ( \mathcal{F} , \mathcal{G} )$ is also an object of $\mathcal{C}$,
     \item[{\bf C2}] every sheaf in $\mathcal{C}$ is acyclic and generated by global sections, and
     \item[{\bf C3}] $\mathcal{C}$ contains the category of locally free $\mathcal{O}_X$-modules.
    \end{itemize}
   \end{Def}
   
   \begin{Thm}[General Serre--Swan {\cite[Theorem~2.1]{Morye}}]\label{Morye}
    Let $(X, \mathcal{O}_X)$ be a locally G-ringed space such that a sequence of coherent sheaves 
    \begin{equation*}
     0 \to \mathcal{E} \to \mathcal{F} \to \mathcal{G} \to 0
    \end{equation*}
    on $X$ is exact if and only if for each $x \in X$ the sequence of $\mathcal{O}_{X,x}$-modules 
    \begin{equation*}
     0 \to \mathcal{E}_x \to \mathcal{F}_x \to \mathcal{G}_x \to 0
    \end{equation*}
    is exact.
    
    Assume that the category of $\mathcal{O}_X$-modules contains an admissible subcategory $\mathcal{C}$ and that every
    locally free sheaf of bounded rank on $X$ is finitely generated by global sections. Then the Serre--Swan theorem holds
    for $(X, \mathcal{O}_X)$, i.\,e., the category of locally free $\mathcal{O}_X$-modules of bounded rank on $X$ 
    and the category of finitely generated projective $\mathcal{O}_X (X)$-modules are equivalent via the global sections functor.     
   \end{Thm}
   Morye's Theorem is formulated for a locally ringed space, but she actually proves the slightly stronger statement
   given here. The site associated to a locally ringed space always has enough points,
   therefore exactness of a short sequence of coherent sheaves may be tested on stalks as in the assumptions of the theorem
   (see \cite[Chapter IX, §3, Proposition~3 on p.~480]{MacLaneMoerdijk94}).

   \begin{Lemma}\label{Papaya}
    Let $X = \bigcup _i (\Sp (A_i))$ be a rigid analytic quasi-Stein variety of bounded dimension, i.\,e.,
    $\sup _{x \in X}{\dim \mathcal{O}_{X,x}} =: d < \infty$. Let $\mathcal{E}$ be a vector bundle of rank $r$ on $X$. 
    Then $\mathcal{E}$ is generated by $(d+1)r$ global sections.
   \end{Lemma}
   \begin{proof}
    By Theorem~A \cite[Satz 2.4.3]{Kiehl67}, every coherent sheaf on a quasi-Stein space is generated by global sections.
    Let $A := \mathcal{O}_X(X)$. We have to show that $\mathcal{E} (X)$ is finitely generated as an $A$-module.
    Gruson does this for Kiehl--Stein spaces which are unions of polydiscs \cite[V, Th{\'e}or{\`e}me 1]{Gruson68}
    and the same proof works for general quasi-Stein spaces. We give it now.
  
    Let $n$ be the rank of the $A_i$-module $\mathcal{E} ( \Sp A_i )$.
    It is independent of $i$. Let $F := A^{n(d+1)}$. Then there exist homomorphisms $F \rightarrow \mathcal{E} (X)$.
    We show that the set $E := \{ F \twoheadrightarrow \mathcal{E} (X) \}$ of epimorphisms from $F$ onto $\mathcal{E} (X)$ is dense
    in $\mathrm{Hom} _A (F, \mathcal{E} (X))$. Then $E$ is in particular nonempty and we know that $\mathcal{E} (X)$ is finitely
    generated.
  
    For any $i$, the map $A_{i+1} \rightarrow A_i$ is flat by Lemma~\ref{flats}. 
    For each $i$, set $F_i := F \tensor _A A_i$ and $M_i := \mathcal{E} (\Sp A_i)$. 
    We have that $F_i \cong F_{i+1} \tensor _A A_i$ and $M_i \cong M_{i+1} \tensor _A A_i$ for all $i$.
    Moreover, $F_i$ and $M_i$ are finitely presented for all $i$.
    Using Lemma~\ref{Lam}, we see that the projective system
    \begin{equation*}
     \cdots \leftarrow A_i \leftarrow A_{i+1} \leftarrow \cdots
    \end{equation*}
    induces a projective system  
    \begin{equation*}
     \cdots \longleftarrow \mathrm{Hom}_{A_i} (F_i, \mathcal{E} ( \Sp A_i )) \longleftarrow 
     \mathrm{Hom}_{A_{i+1}} (F_{i+1}, \mathcal{E} ( \Sp A_{i+1} )) \longleftarrow \cdots
    \end{equation*}
    The limit is $\mathrm{Hom} _A (F, \mathcal{E} (X))$ by Kiehl's Theorem~B \cite[Satz 2.4.2]{Kiehl67}.
  
    As flat base change maps epimorphisms to epimorphisms, we can look at the projective system
    \begin{equation*}
     \cdots \longleftarrow E_i \longleftarrow E_{i+1} \longleftarrow \cdots
    \end{equation*}
    where $E_i :=  \{ F_i \twoheadrightarrow \mathcal{E} (\Sp A_i) \}$. We claim that its limit is $E$.
    It is clear that $E \subset \operatorname{lim} (E_i)_{i \in \mathbb{N}}$.
    We have to show that the limit of a system $(\phi _i)_i \in (E_i)_i$ of epimorphisms is an epimorphism:
    $\mathrm{lim} (\phi _i)_i \in E$.
    As the category of coherent sheaves on a rigid analytic variety is abelian \cite[p.~87]{FresnelPut}, 
    the sheaf cokernel $\mathcal{D}$ of $\mathcal{F} \rightarrow \mathcal{E}$ is coherent. 
    By Kiehl's Theorem~B \cite[Satz 2.4.2]{Kiehl67},
    it is enough to test surjectivity of a map of coherent sheaves on an affinoid covering. 
    But we have $\mathcal{D}( \Sp A_i) = 0$ for all $i$, because there
    the morphisms $\mathcal{F} (\Sp A_i) \twoheadrightarrow \mathcal{E} (\Sp A_i)$ are epi. Thus, the cokernel sheaf
    $\mathcal{D}$ is zero and we have $\mathrm{lim} (\phi _i)_i \in E$.
  
    A lemma by Gruson \cite[V, Lemme~2]{Gruson68} states that if $B$ is any Banach algebra of topologically finite type of Krull dimension $d$
    and if $P$ is a projective $B$-module of rank $p$ and $L$ is the free $B$-module $B^{(d+1)p}$, then the set of epimorphisms
    $L \rightarrow P$ is open and dense in the topological $B$-module $\mathrm{Hom} (L,P)$.
  
    This applies in particular to the affinoid algebras $A_i$.
    We get that $E_i$ is open and dense in $\mathrm{Hom}_{A_i} (F_i, \mathcal{E} ( \Sp A_i ))$ for each $i$.
    By the general Mittag-Leffler Theorem \cite[II.3.5, Th{\'e}or{\`e}me~1]{BourbakiTop1}, 
    $E$ is dense in $\mathrm{Hom}_A (F, \mathcal{E} (X))$ with respect to the limit topologies. 
    Hence $E$ is nonempty and $\mathcal{E} (X)$ is generated by $(d+1)n = (d+1)r$ global sections.
   \end{proof}
   
   \begin{Thm}[rigid analytic Serre--Swan]\label{rigSerreSwan}
    Let $X = \bigcup _i (\Sp (A_i))$ be a quasi-Stein variety of bounded dimension, i.\,e.,
    $\sup _{x \in X}{\dim \mathcal{O}_{X,x}} < \infty$. Set $A := \mathcal{O}_X (X)$. Then the global sections functor
    \begin{equation*} 
     \Gamma\colon \operatorname{VB} (X) \longrightarrow \operatorname{A-Mod}^{\mathrm{fin}, \mathrm{proj}}
    \end{equation*}
    from vector bundles over $X$ to finitely generated projective $A$-modules
    is an equivalence of categories.
   \end{Thm}
   \begin{proof}
    Let us verify the assumptions of Theorem~\ref{Morye} (general Serre--Swan). 
    \begin{itemize}
     \item A sequence of coherent sheaves 
        \begin{equation*}
         0 \to \mathcal{E} \to \mathcal{F} \to \mathcal{G} \to 0
        \end{equation*}
        on a rigid analytic variety $X$ is exact if and only if for each $x \in X$ the sequence of $\mathcal{O}_{X,x}$-modules 
        \begin{equation*}
         0 \to \mathcal{E}_x \to \mathcal{F}_x \to \mathcal{G}_x \to 0
        \end{equation*}
        is exact \cite[p.~192]{FresnelPut}.
     \item The category of coherent $\mathcal{O}_X$-modules on $X$ is abelian \cite[p.~87]{FresnelPut} and  
        an admissible subcategory of $\category{Mod}_{\mathcal{O}_X}$\cite[Satz~2.4]{Kiehl67}.
     \item Vector bundles are finitely generated by global sections. 
        This is true for affinoid varieties by Kiehl \cite[Satz 2.2]{Kiehl67}
        and for general quasi-Stein varieties of bounded dimension by Lemma~\ref{Papaya}.
    \end{itemize}
    Consequently the Serre--Swan Theorem holds for rigid analytic quasi-Stein varieties.
   \end{proof}
   \begin{Rem}
    The rigid analytic Serre--Swan theorem and its proof have also been indicated by Kohlhaase \cite[Proposition~A.2]{Kohlhaase11}.
    If $X = \Sp A$ is affinoid itself, it follows directly from Kiehl's work~\cite{Kiehl67}, as was pointed out by
    Gruson \cite{Gruson68} and L\"{u}tkebohmert \cite{Lutke77}. If 
    \begin{equation*}
     X = \bigcup _{s<r} \mathbb{B}^n_s = \bigcup _{\vert \eta \vert = s<r} k \langle \eta^{-1} T_1, \dotsc , \eta^{-1} T_n \rangle
    \end{equation*}
    is the ``open disc'' of radius $r \leq \infty$, the statement is \cite[V, Th{\'e}or{\`e}me 1]{Gruson68}. 
    Gruson points out that the same proof should work for general Kiehl--Stein varieties.
   \end{Rem}
 
   \begin{Thm}[algebraic Serre--Swan]\label{ZarSerreSwan}
    Let $X = \Sp A$ be an affinoid variety. The global sections functor
    \begin{equation*}
     \Gamma\colon \operatorname{VB}_\mathrm{Zar} (X) \longrightarrow \operatorname{A-Mod}^{\mathrm{fin}, \mathrm{proj}}
    \end{equation*}
    from Zariski vector bundles over $X$ to finitely generated projective $A$-modules
    is an equivalence of categories. 
   \end{Thm}
   \begin{proof}
    Again we need to verify the assumptions of Theorem~\ref{Morye} (general Serre--Swan). 
    \begin{itemize}
     \item A sequence of coherent sheaves 
        \begin{equation*}
         0 \to \mathcal{E} \to \mathcal{F} \to \mathcal{G} \to 0
        \end{equation*}
        on an affinoid variey $X$ is exact if and only if for each $x \in X$ the sequence of the Zariski stalks 
        \begin{equation*}
         0 \to \mathcal{E}^{\mathrm{Zar}}_x \to \mathcal{F}^{\mathrm{Zar}}_x \to \mathcal{G}^{\mathrm{Zar}}_x \to 0
        \end{equation*}
        is exact \cite[tag~00HN]{stacks-project}.
     \item As $(X, \mathrm{Zar} , \mathcal{O}_X^{\mathrm{Zar}})$ is a locally ringed space (not only locally G-ringed),
        the category of coherent Zariski $\mathcal{O}_X$-modules on $X$ is abelian 
        by the classical arguments \cite[tag~01BY]{stacks-project}. 
        It is an admissible subcategory of $\category{Mod}_{\mathcal{O}_X}$:
        \begin{itemize}
         \item[{\bf C1}] Internal Hom between coherent sheaves is coherent \cite[tag~01CQ]{stacks-project}.
         \item[{\bf C2}] Every coherent sheaf is acyclic by a theorem by Grothendieck (see \cite[Theorem~2.7]{Hartshorne})
             and generated by global sections by noetherianity (cf.~the argument in the next point below). 
         \item[{\bf C3}] is clear.
        \end{itemize}
     \item\label{Zarfinglob} Finite generation by global sections. 
        Let $\mathcal{E}$ be a locally free sheaf of bounded rank $n$ on $X$ and $\operatorname{dim} X = d$.
        As $A$ is noetherian, $\mathcal{E}$ has a finite local trivialisation of the form $\{ \Sp (A_{f_i}) \} _{i=1, \dotsc , N}$
        with $f_i \in A$ for all $i$. That is, $\{ \Sp (A_{f_i}) \} _{i=1, \dotsc , N}$ is a finite covering of $X$
        by affinoid domains $X_i := \Sp (A_{f_i})$ such that for all $i$ the restriction $\mathcal{E} (X_i)$ is trivial.
        In other words,
        \begin{equation*}
         \mathcal{E} (X_i) \cong (A_{f_i})^n.
        \end{equation*}
        Let $f_{i_1}, \dotsc , f_{i_n} \in A_{f_i}$ be generators of $\mathcal{E} (X_i)$ as an $A_{f_i}$-module.
        Now if some $f_{i_j}$ is not in the image of the localisation $A \subset A_{f_i}$, this means that $f_{i_j}$ has
        a finite power $s_j$ of $f_i$ in the denominator.
        The elements $f_{i_1} f_i^{s_1} , \dotsc , f_{i_n} f_i^{s_n}$ are now in the image of the localisation $A \subset A_{f_i}$.
        As $f_i$ is a unit in $A_{f_i}$, the $f_{i_1} f_i^{s_1} , \dotsc , f_{i_n} f_i^{s_n}$ still generate 
        $\mathcal{E} (X_i)$ as an $A_{f_i}$-module. They lift to global sections
        $\tilde{f}_{i_1}, \dotsc , \tilde{f}_{i_n} \in A$ and
        hence $\mathcal{E}$ is generated by the global sections 
        $\tilde{f}_{1_1}, \dotsc , \tilde{f}_{1_n}, \dotsc , \tilde{f}_{N_1}, \dotsc , \tilde{f}_{N_n}$.
        By induction over the dimension we can even show that $N$ can be chosen less than or equal to $d + 1$.
    \end{itemize}
  \end{proof}
  
  \begin{Cor}\label{vbequiv}
   The categories of analytic vector bundles and of Zariski vector bundles on an affinoid variety $X$
   coincide as subcategories of the category of sheaves on $X$.
  \end{Cor}  
   
  Recall that a rigid analytic variety is called \emph{quasicompact} if it can be covered by finitely many affinoids.
  \begin{Prop}\label{GfinG}
   Every vector bundle on a quasicompact rigid analytic variety has a finite local trivialisation.
  \end{Prop}
  \begin{proof}
   Let $X$ be a quasicompact variety and let $\mathcal{E}$ be a vector bundle on $X$. Assume that $X$ is affinoid. Then $\mathcal{E}$ is
   isomorphic to a Zariski vector bundle $\mathcal{E}'$ on $X$. It has a finite local trivialisation because $X$ is noetherian.
   If $X$ is not affinoid, it has a finite admissible covering $\{ U_1 , \dotsc , U_r \}$ by affinoid subsets. The restrictions of
   $\mathcal{E}$ to those affinoid subsets each have a finite local trivialisation. Putting them together yields a finite local
   trivialisation of $\mathcal{E}$.
  \end{proof}
  
  The same holds for quasi-Stein varieties:
  \begin{Thm}[Ben's Theorem]\label{SteinfinG}
   Every vector bundle on a rigid analytic quasi-Stein variety has a finite local trivialisation.
  \end{Thm}
  The author asked for a proof or counterexample on  
  \href{http://mathoverflow.net/questions/219140/trivialisation-of-vector-bundles-on-stein-spaces}{\protect\nolinkurl{mathoverflow.net} (question 219140)}. 
  The following proof was given by the user Ben. He proved that, over a locally ringed space, any vector bundle which
  is finitely generated by global sections admits a finite local trivialisation. This easily carries over to a rigid analytic
  variety with the G-topology.
  \begin{proof}
   Let $X$ be a quasi-Stein variety and $\mathcal{E}$ be a vector bundle of rank $r$ on $X$. By Lemma~\ref{Papaya}, 
   it is generated by finitely many global sections $s_1, \dotsc , s_d$. 
   Let $X_B$ be the Berkovich space associated to $X$. As $X$ is quasi-Stein, every $x \in X_B$ has an affinoid
   neighbourhood. By Proposition~\ref{vbequivBer}, the category of vector bundles with respect to the
   Berkovich topology on $X_B$ and the category of analytic vector bundles on $X$ are equivalent. 
   The proof is done by explicit construction.
   
   For $1 \leq j \leq r$, let $e_j$ be the section of $\mathcal{O}^r_{X_B}$ that is given by the constant function 
   $(0, \dotsc , 0,1,0, \dotsc , 0)$ where the $1$ is located in the $j$-th entry.
   Clearly, $e_1, \dotsc , e_r$ are generators of $\mathcal{O}^r_{X_B}$.
   For each 
   \begin{equation*}
    I = \{ i_1 , \dotsc , i_r \} \subset \{ 1, \dotsc , d \} 
   \end{equation*}
   define the morphism
   \begin{align*}
    \varphi _I \colon \mathcal{O}^r_{X_B} &\longrightarrow \mathcal{E} \\
                e_j &\longmapsto s_{i_j} .
   \end{align*}
   Then the subsets
   \begin{equation*}
    U_I = \{ x \in {X_B} \mid (\varphi _I )_x \to \mathcal{E}_x \text{ is an iso } \}
   \end{equation*}
   form a finite trivialisation of $\mathcal{E}$.
   By definition of $U_I$, the map 
   \begin{equation*}
    \varphi _I |_{U_I} \colon \mathcal{O}^r_{X_B} ( U_I) \to \mathcal{E} (U_I)
   \end{equation*}
   induces isomorphisms on stalks. 
   As $\mathcal{E}$ is coherent, it has to be an isomorphism.
   
   We still need to check that $\{ U_I \mid I = \{ i_1 , \dotsc , i_r \} \subset \{ 1, \dotsc , d \}  \}$
   is a covering of ${X_B}$. Let $x \in {X_B}$. The stalk $\mathcal{E}_x$ is generated by the residue
   classes of $s_1, \dotsc , s_d$. As it is of dimension $r$, already $r$ of them suffice. Thus, there is an $I$ such that
   $x \in U_I$. Hence, all the $U_I$ together cover ${X_B}$. Each $U_I$ is open because the set of $x \in {X_B}$ for which
   the residue classes of $s_{i_1}, \dotsc , s_{i_r}$ in $\mathcal{E}_x$ satisfy a non-trivial relation is closed.
   We have shown that every vector bundle over the Berkovich space associated to a quasi-Stein rigid analytic variety
   has a finite local trivialisation with respect to the Berkovich topology. 
   We need a finite local trivialisation with respect to the G-topology.
   
   Let $X_1 \subset X_2 \subset \dotsb$ be a strictly affinoid covering of $X_B$.
   
   Let $\mathcal{V} \in \operatorname{Cov} ( X_B)$ be an admissible G-covering which is a local trivialisation
   of $\mathcal{E}$ such that for any $V \subset \mathcal{V}$ there exists an $I$ 
   with $V \subset U_I$.
   Let $\mathcal{W} \subset \mathcal{V}$ be a subcovering whose restriction to each $X_i$ is finite.
   This exists because the $X_i$ are affinoid and hence compact.
   Set
   \begin{equation*}
    \tilde{U}_I := \bigcup _{\substack{V \in \mathcal{W}, \\ V \subset U_I}} V.
   \end{equation*}
   Now $\{ \tilde{U}_I \mid I = (i_1, \dotsc , i_r) \subset \{ 1 , \dotsc , d \} \}$
   is a finite admissible G-covering of $X$ which trivialises $\mathcal{E}$.
   It gives rise to a finite local trivialisation of the original vector bundle over the rigid analytic space $X$.
  \end{proof}
  
 \section{Serre's problem and the Bass--Quillen conjecture}\label{secSerre}
  In \cite{SerreFAC}, Serre asked whether all finitely generated projective modules over a polynomial ring over a field are free.
  The question was answered affirmatively and independently in 1976 by Quillen \cite{Quillen76} and by Suslin \cite{Suslin76}. 
  The theorem is now known as the Quillen--Suslin Theorem. We warmly recommend Lam's wonderful book \cite{Lam06} about Serre's problem.
  
  The analogue of the Quillen--Suslin Theorem in rigid analytic geometry is also true:
  \begin{Thm}[{Lütkebohmert \cite[Theorem~1]{Lutke77}, Kedlaya \cite[Proposition~6.6]{Kedlaya04}}]\label{rigQuillenSuslin}
   Every finitely generated projective \mbox{$k \langle T_1 , \dotsc , T_n \rangle$-module} is free.
  \end{Thm}
  
  When Serre's problem was still open, Bass generalised it to the following:
  \begin{Quest}[Bass {\cite[Problem IX on page 21]{Bass73}}]\label{BassQuillen}
   Let $R$ be a commutative regular ring. Is every finitely generated projective $R[T]$-module extended from $R$,
   i.\,e., given an $R[T]$-module $M$, does there exist an $R$-module $N$ such that
   $M \cong N \otimes _R R[T]$?
  \end{Quest}  

  Using Quillen Patching, it is enough to prove the statement for a commutative regular local ring $R$. 
  This problem is known as the \emph{Bass--Quillen Conjecture}.
  Hartmut Lindel showed that the Bass--Quillen Conjecture is true if $R$ is a regular algebra over a field $K$
  such that $R$ is essentially of finite type over $K$ \cite{Lindel81}.
  Lindel's theorem was generalised by Dorin Popescu \cite[Corollary~4.4]{Popescu89} and others, 
  e.\,g.~Sankar Prasad Dutta \cite[Theorem~3.4]{Dutta95}. 
  For the most general form of the theorem, see \cite[Theorem~5.2.1]{AsokHoyoisWendt15}.

  We are concerned with the corresponding question in the rigid analytic world:
  \begin{Quest}\label{QuestBassQuillen}
   Let $A$ be a smooth affinoid $k$-algebra and $M$ a finitely generated projective
   $A \langle T \rangle$-module. Is $M$ extended from $A$? That is, does there exist
   a finitely generated projective $A$-module $N$ such that
  \begin{equation*}
   M \cong N \tensor _A A \langle T \rangle ?
  \end{equation*}
  \end{Quest}
  
 \subsection{Homotopy invariance}
  We generalise Question \ref{QuestBassQuillen} to:
  \begin{Quest}\label{Questgeneral}
   Let $X$ be a smooth rigid analytic $k$-variety and $I \in \{ \mathbb{B}^1 , \mathbb{A}^1_\mathrm{rig} \}$.
   Denote by
   \begin{equation*}
    \operatorname{pr}_1 \colon X \times I \longrightarrow X
   \end{equation*}
   the projection onto the first factor. Is every vector bundle $\mathcal{E}$ over $X \times I$ the pullback
   along $\operatorname{pr}_1$ of some vector bundle $\mathcal{F}$ over $X$?
  \end{Quest}
  In the sequel we will refer to this as \emph{$I$-invariance}\index{$I$-invariance}.
  For affinoid $X$ and $I = \B ^1$ this is equivalent to Question \ref{QuestBassQuillen} by the Serre--Swan Theorem.
  
  The first thing to check is whether the Quillen--Suslin Theorem still holds for $I = \A ^1$. 
  That is, are all vector bundles over $\mathbb{A}^1_\mathrm{rig}$ trivial?
  \begin{Prop}[{\cite[V, Proposition~2]{Gruson68}}]\label{QuillenSuslinopendisc}
   For $r \in ( \vert k \vert \cup \{ \infty \} )^n$ let $X_r$ be an open polydisc of polyradius $r$, that is
   \begin{equation*}
    X_r = \bigcup _{\vert \eta _i \vert = s_i < r_i \forall i} 
    \operatorname{Sp} ( k \langle \eta _1^{-1} T_1 , \dotsc , \eta _n^{-1} T_n \rangle ).
   \end{equation*}
   The Picard group $\operatorname{Pic} (X_r)$ of line bundles over $X_r$ is trivial if and only if one of the following holds:
   \begin{enumerate}
    \item the field $k$ is spherically complete or
    \item the polyradius is $r = ( \infty , \dotsc , \infty )$, that is, $X_r = \mathbb{A} ^n_\mathrm{rig}$ is the analytic affine space.
   \end{enumerate}
  \end{Prop}
  The ``only if'' direction was proven by Lazard \cite[Proposition~6]{Lazard62}: Let $k$ be spherically incomplete and $r \in \vert k \vert$.
  Lazard constructs a divisor on $X_r$ which is not a principal divisor. This implies that open discs $X_r$ which are bounded in 
  at least one direction have non-trivial line bundles \cite[p.~87, Remarque~2]{Gruson68}.
  
\subsection{$\mathbb{A}^1_\mathrm{rig}$-invariance}
 We prove that the Picard group is $\mathbb{A}^1_\mathrm{rig}$-invariant on all smooth $k$-rigid analytic varieties.
 We state the proposition first, then collect the ingredients that we need in the proof, then present the proof.
 \begin{Prop}\label{rigA1inv}\index{$\mathbb{A}^1_\mathrm{rig}$-invariance}
  Let $k$ be a complete non-archimedean, non-trivially valued field and
  $X$ be a smooth $k$-rigid analytic variety. Then
  \begin{equation*}
   \operatorname{Pic} (X) \cong \operatorname{Pic} (X \times \mathbb{A}^1_\mathrm{rig} ).
  \end{equation*}  
 \end{Prop}
 
 The next theorem is the main ingredient in the proof of Proposition~\ref{rigA1inv}.
 Let us fix the notation. Let $X = \operatorname{Sp} A$ be a smooth affinoid $k$-variety.
 Let $\pi \in k^\times$ with $\vert \pi \vert < 1$. Define $r := \vert \pi \vert ^{-1}$ and
 \begin{equation*}
  \mathbb{B} ^1_{r^j} = \operatorname{Sp} (k \langle \pi ^j T \rangle ).
 \end{equation*}
 The inductive system
 \begin{equation*}
  \big( \dotsb \hookrightarrow X \times \mathbb{B}^ 1_{r^j} \hookrightarrow X \times \mathbb{B}^ 1_{r^{j+1}} 
  \hookrightarrow \dotsb \big) _{j \in \mathbb{N}}
 \end{equation*}
 induces a projective system of groups
 \begin{equation*}
  \big( \dotsb \rightarrow \operatorname{Pic} ( X \times \mathbb{B}^ 1_{r^j+1} ) \rightarrow 
  \operatorname{Pic} ( X \times \mathbb{B}^ 1_{r^j} ) \rightarrow \dotsb \big) _{j \in \mathbb{N}} ,
 \end{equation*}
 called a \emph{pro-group}\index{pro-group} and denoted by 
 $\underset{j \in \mathbb{N}}{\operatorname{``lim''}} \operatorname{Pic} ( X \times \mathbb{B}^ 1_{r^j} )$.
 \begin{Thm}[Kerz--Saito--Tamme {\cite[Theorem~4]{KerzSaitoTamme16}}]\label{ThmTamme}
  Let $X = \operatorname{Sp} A$ be a smooth affinoid $k$-variety.
  Let $\pi \in k^\times$ with $\vert \pi \vert < 1$. Define $r := \vert \pi \vert ^{-1}$ and
  $\mathbb{B} ^1_{r^j}$ as above.
      
  Then the projection $X \times \mathbb{A}^1_\mathrm{rig} \to X$
  induces an isomorphism of pro-groups
  \begin{equation*}
   \operatorname{Pic} (X) \cong \underset{j \in \mathbb{N}}{\operatorname{``lim''}} \operatorname{Pic} ( X \times \mathbb{B}^ 1_{r^j} ).
  \end{equation*}
 \end{Thm}
  
 \begin{Cor}\label{affinoidA1inv}
  Let $X = \operatorname{Sp} A$ be a smooth affinoid variety.
  Then the projection $X \times \mathbb{A}^1_\mathrm{rig} \to X$
  induces an isomorphism
  \begin{equation*}
   \operatorname{Pic} (X) \cong \operatorname{Pic} ( X \times \mathbb{A}^ 1_\mathrm{rig} ).
  \end{equation*}
 \end{Cor}
 \begin{proof}
  The isomorphism of pro-systems in Theorem~\ref{ThmTamme} induces an isomorphism of their limits.
  Applying the Serre--Swan Theorem~\ref{rigSerreSwan}, we may replace line bundles over $X$, respectively,
  $X \times \mathbb{B}^ 1_{r^ j}$, respectively, $X \times \mathbb{A}^ 1_\mathrm{rig}$ by invertible modules
  over $A$, respectively over $A \langle \pi ^ j T \rangle$, respectively over $\mathcal{O} (X \times \mathbb{A}^ 1_\mathrm{rig} )$.
  We need to show that 
  \begin{equation*}
   \operatorname{lim}_j \operatorname{Pic} (A \langle \pi ^ j T \rangle )
    = \operatorname{Pic} ( \operatorname{lim} _j A \langle \pi ^ j T \rangle ).
  \end{equation*}
  This means, we have to show that
  \begin{equation*}
   \operatorname{lim} ^{(1)} \big( \operatorname{Pic} ( A \langle \pi ^ j T \rangle )_j \big) = 0.
  \end{equation*}
  This holds by Kiehl's Theorem~B.
 \end{proof}
 
 \begin{Cor}\label{A1local}
  Let $X = \operatorname{Sp} A$ be a smooth $k$-rigid analytic variety.
  Then every line bundle over $X \times \mathbb{A}^ 1_\mathrm{rig}$ has a local trivialisation of the form 
  $\{ U_i \times \mathbb{A}^1_\mathrm{rig} \} _ {i \in I}$ such that $\{ U_i \} _{i \in I}$ is an admissible covering of $X$.
 \end{Cor}
 \begin{proof}
  Let $\{ V_j \} _{j \in J}$ be an admissible covering of $X$ by affinoid subsets. Let $\mathcal{L}$ be a line bundle 
  over $X \times \mathbb{A}^ 1_\mathrm{rig}$. 
  As usual, consider the projection
  \begin{equation*}
   \operatorname{pr}_1 \colon X \times \mathbb{A}^ 1_\mathrm{rig} \longrightarrow X.
  \end{equation*}
  By Corollary~\ref{affinoidA1inv}, there exist line bundles $\mathcal{L}_j$
  over $X$ such that for each $j \in J$, we have
  \begin{equation*}
   \mathcal{L} (V_j \times \mathbb{A}^ 1_\mathrm{rig} ) \cong \operatorname{pr}_1^\ast \mathcal{L}_j.
  \end{equation*}
  For each $j$, let $\{ U_{i_j} \} _ {i_j\in I_j}$ be a local trivialisation of $\mathcal{L}_j$.
  Then $\{ U_{i_j} \times \mathbb{A}^1_\mathrm{rig} \} _ {i_j\in I_j}$ is a local trivialisation of 
  $\mathcal{L} (V_j \times \mathbb{A}^ 1_\mathrm{rig} )$. Setting $I:= \bigcup _{j \in J} I_j$,
  the set $\{ U_i \times \mathbb{A}^1_\mathrm{rig} \} _ {i \in I}$ provides a local trivialisation of $\mathcal{L}$
  of the desired form.
 \end{proof}
 
 \begin{Lemma}\label{cocycA1}
  Let $A$ be a reduced affinoid algebra. Then
  \begin{equation*}
   \mathcal{O} (\operatorname{Sp} (A) \times \mathbb{A}^1_\mathrm{rig} ) ^\times \cong A ^\times .
  \end{equation*}
  In particular,
  \begin{equation}\label{equationA1inv}
   \mathcal{O} (\mathbb{A}^1_\mathrm{rig} )^\times = k^\times .
  \end{equation}
 \end{Lemma}
 \begin{proof}
  In the proof we consider an invertible formal power series $f  \in A[[T]]$. 
  We compute the condition on the coefficients of the power series $f$ and its formal inverse $f^{-1}$ 
  for $f$ and $f^{-1}$ to converge on a tubular neighbourhood 
  around the zero section 
  \begin{equation*}
   \operatorname{Sp} A \times \{ 0 \} \subset \operatorname{Sp} (A) \times \mathbb{A}^1_\mathrm{rig}.
  \end{equation*}
  The condition is that the coefficients' norms become small as the radius of the tubular neighbourhood becomes big.
  We conclude that if both $f$ and $f^{-1}$ are to converge on the whole space
  $\operatorname{Sp} (A) \times \mathbb{A}^1_\mathrm{rig}$, then all coefficients of $f$ and $f^ {-1}$ except the
  constant coefficient have to vanish.
  
  Let 
  \begin{equation*}
   f = \sum _{i=0}^\infty f_i T^i \in \mathcal{O} (\operatorname{Sp} (A) \times \mathbb{A}^1_\mathrm{rig} ) \subset A [[T]],
   \quad \text{all } f_i \in A.
  \end{equation*}
  As a formal power series, $f$ is invertible if and only if its constant coefficient $f_0$ is invertible in $A$.
  Assume without loss of generality that $f = 1 + \sum _{i = 1}^\infty f_i T^i$. The formal power series inverse
  of $f$ is
  \begin{equation*}
   f^{-1} = 1 + \sum _{i=1}^\infty \tilde{f}_i T^i 
  \end{equation*}
  where
  \begin{equation*}
   \tilde{f}_i = - \sum _{j = 1}^i f_i \tilde{f}_{i-j}.
  \end{equation*}
  We want
  \begin{equation*}
   f, f^{-1} \in \mathcal{O} (\operatorname{Sp} (A) \times \mathbb{A}^1_\mathrm{rig} )
  \end{equation*}
  which means
  \begin{equation}\label{teaquation}
   f, f^{-1} \in A \langle \eta ^{-l} T \rangle 
  \end{equation}
  for some $\eta \in k$ with $\vert \eta \vert > 1$ and for all $l \in \mathbb{N}$. Fix $l \in \mathbb{N}$ and set $r := \vert \eta \vert ^l$.
  For this $l$, the relation \eqref{teaquation} holds if and only if the coefficients satisfy
  \begin{equation}\label{keayboard}
   \Vert f _i \Vert r^i < 1 \quad \forall i \geq 1.
  \end{equation}
  
  Let us check this claim. As $A$ is reduced, its norm $\Vert \cdot \Vert$ is power-multiplicative.
  Assume that there exists an $i \geq 1$ with $\Vert f_i \Vert r^i \geq 1$. Choose
  $i$ minimal with respect to this property. Then
  \begin{equation*}
   \Vert \tilde{f}_i \Vert r^i = \bigg\Vert f_i - \sum _{j = 1}^{i-1} f_j \tilde{f}_{i-j} \bigg\Vert r^i = \Vert f_i \Vert r^i \geq 1
  \end{equation*}
  because for $1 \leq j \leq i-1$ we have
  \begin{equation*}
   \Vert f_j \tilde{f}_{i-j} \Vert r^i \leq \Vert f_j \Vert \Vert \tilde{f}_{i-j} \Vert r^i \underset{r \geq 1}{\leq}
   \Vert f_j \Vert r^i \cdot \Vert \tilde{f}_{i-j} \Vert r^i 
   \underset{\ast}{<} 1 \leq \Vert f_i \Vert r^i.
  \end{equation*}
  The inequality marked with an asterisk $\ast$ holds because $i$ was chosen such that $\Vert f _j \Vert r^j < 1$ for all $j < i$.
  Similarly, we get
  \begin{equation*}
   \Vert \tilde{f}_{2i} \Vert r^{2i} = \bigg\Vert \sum _{j=1}^{2i} f_j \tilde{f}_{2i-j} \bigg\Vert r^{2i} =
   \Vert f_i \tilde{f}_i \Vert r^{2i} 
   = ( \Vert f_i \Vert r^ i )^2 \geq 1
  \end{equation*}
  and for all $n \in \mathbb{N}$
  \begin{equation*}
   \Vert \tilde{f}_{ni} \Vert r^{ni} \geq 1.
  \end{equation*}
  Thus, $f^{-1} \notin A\langle \eta ^{-l} T \rangle$, which contradicts the assumptions. We have proven that 
  \eqref{teaquation} and \eqref{keayboard} are equivalent for a fixed $l$.
  
  To get back to our problem, we have
  \begin{equation*}
   f \in A \langle \eta ^{-l} T \rangle ^\times \quad \forall l \in \mathbb{N}
  \end{equation*}
  which is equivalent to
  \begin{equation*}
   \Vert f_i \Vert < \frac{1}{r^i} = \frac{1}{\vert \eta \vert ^{il}} \quad \forall i \geq 1 \enspace \forall l \in \mathbb{N}.
  \end{equation*}
  As $\vert \eta \vert > 1$, this means that
  \begin{equation*}
   f \in \mathcal{O} \big( \operatorname{Sp} (A) \times \mathbb{A}^1_\mathrm{rig} \big) ^\times \quad \Longleftrightarrow \quad \left\{
   \begin{split}
    f_0 &\in A^ \times &\enspace &\text{and} \\
    f_i &= 0 &&\text{for } i \geq 1.
   \end{split} \right. 
  \end{equation*}
 \end{proof}

 \begin{proof}[Proof of Proposition~\ref{rigA1inv}]
  Let $X$ be a smooth $k$-rigid analytic variety. Let $\mathcal{L}$ be a line bundle over $X \times \mathbb{A}^1_\mathrm{rig}$.
  By Corollary~\ref{A1local}, the bundle $\mathcal{L}$ has a local trivialisation of the form 
  $\{ U_i \times \mathbb{A}^1_\mathrm{rig} \} _ {i \in I}$ such that $\{ U_i \} _{i \in I}$ is an admissible covering of $X$.
  For $i,j \in I$ and $U_i \cap U_j \neq \emptyset$, the transition function on $(U_i \cap U_j) \times \mathbb{A}^1_\mathrm{rig}$
  is given by an element $f \in \mathcal{O}_{X \times \mathbb{A}^1_\mathrm{rig}} ((U_i \cap U_j) \times \mathbb{A}^1_\mathrm{rig})$.
  By Lemma~\ref{cocycA1}, this function is in fact an element of $\mathcal{O}_X (U_i \cap U_j)$.
 \end{proof}

  \begin{Ex}\label{Exrk2P1}
   We show that neither $\mathbb{B}^1$- nor $\mathbb{A}^1_\mathrm{rig}$-invariance hold for vector bundles of rank at least two 
   over $\mathbb{P}^1_\mathrm{rig}$. Every counterexample to $\mathbb{A}^1$-invariance of algebraic vector bundles over the
   algebraic projective line works. We present the analytified version of {\cite[Example 3.2.9]{AsokMorel11}}.
   
   Introducing coordinates $(t,x)$ on $(\mathbb{P}_\mathrm{rig} ^1 \setminus \{ 0, \infty \} )\times \mathbb{A}_\mathrm{rig} ^1$,
   let $\mathcal{E}$ be the vector bundle on $\mathbb{P}_\mathrm{rig} ^1 \times \mathbb{A}_\mathrm{rig} ^1$
   which is free on $(\mathbb{P}_\mathrm{rig} ^1 \setminus \{ 0 \} )\times \mathbb{A}_\mathrm{rig} ^1$ and
   free on $(\mathbb{P}_\mathrm{rig} ^1 \setminus \{ \infty \} )\times \mathbb{A}_\mathrm{rig} ^1$
   and has the transition matrix 
   $\left( \begin{smallmatrix}
     t^2 & xt \\
     0 & 1
    \end{smallmatrix} \right)$
   on $(\mathbb{P}_\mathrm{rig} ^1 \setminus \{ 0, \infty \} )\times \mathbb{A}_\mathrm{rig} ^1$.
   We show that it is not the pullback of any bundle on $\mathbb{P}^1_\mathrm{rig}$.
   Restricting $\mathcal{E}$ to $0 \in \A ^1$ yields a bundle $\mathcal{E}_0$ on $\mathbb{P}^1_\mathrm{rig}$
   described by 
   $\left( \begin{smallmatrix}
     t^{-1} & 0 \\
     0   & t
   \end{smallmatrix} \right)$.
   Restricting to $1 \in \A ^1$ we get $\mathcal{E}_1$, described by
   $\left( \begin{smallmatrix}
     t^2 & t \\
     0 & 1
    \end{smallmatrix} \right)$.
  
   The bundles $\mathcal{E}_0$ and $\mathcal{E}_1$ are the analytifications of bundles $\widetilde{\mathcal{E}}_0$, 
   respectively $\widetilde{\mathcal{E}}_1$ on the algebraic variety $\mathbb{P}^1$ that are described by the same transition matrices. 
   By GAGA \cite[Proposition~9.3.1]{FresnelPut},
   the analytic bundles $\mathcal{E}_0$ and $\mathcal{E}_1$ are isomorphic if and only if the algebraic bundles
   $\widetilde{\mathcal{E}}_0$ and $\widetilde{\mathcal{E}}_1$ are isomorphic. 
   By Grothendieck's classification of line bundles over $\PR ^1$ \cite{Grothendieck57}, they are not. 
   (See Hazewinkel--Martin \cite[Theorem~4.1]{HazewinkelMartin81} for the case of a base field different from $\C$.)
 
   Denote the projection, the zero section and the one section by 
   \begin{align*}
    \operatorname{pr}_1 &\colon \mathbb{P}^1_\mathrm{rig} \times \mathbb{A}^1_\text{rig} \rightarrow \mathbb{P}^1_\mathrm{rig}, 
    \qquad \text{respectively}\\
    \iota _x &\colon \mathbb{P}^1_\mathrm{rig} \rightarrow \mathbb{P}^1_\mathrm{rig} \times \{ x \} \subset 
    \mathbb{P}^1_\mathrm{rig} \times \mathbb{A}^1_\mathrm{rig} \enspace \text{for} \enspace x \in \{ 0,1 \}.
   \end{align*}
   If we had $\mathcal{E} \cong \operatorname{pr}_1 ^{\ast} \mathcal{F}$ for some bundle $\mathcal{F}$
   on $\mathbb{P}^1_\mathrm{rig}$, we would have
   \begin{equation*}
    \mathcal{E}_0 = \iota _0 ^ \ast \mathcal{E} = \iota _0^\ast \operatorname{pr}_1 ^\ast \mathcal{F} \cong \mathcal{F} \cong 
    \iota _1 ^ \ast \operatorname{pr}_1 ^ \ast \mathcal{F} \mathcal{E}_1 = \iota _1^\ast \mathcal{E}.
   \end{equation*}
   But we just checked that $\mathcal{E}_0 \ncong \mathcal{E}_1$.
  \end{Ex}
   
   \begin{Ex}[Gerritzen {\cite[§3.3]{Gerritzen77}}]\label{Excusp1}
    There exists an affinoid variety with a cusp singularity such that
    $\mathbb{B}^1$-invariance does not hold for vector bundles 
    over this affinoid variety. We expect that $\A ^1_\text{rig}$-invariance does not hold either. 
    This question is addressed in work in progress.
    The space
    \begin{equation*}
     X = \operatorname{Sp} k \langle T_1 , T_2 \rangle / (T_1^2 - T_2 ^3)
    \end{equation*}
    is an affinoid neighbourhood of
    the cusp of the analytification of $\operatorname{Spec} k [T_1 ,T_2 ] / (T_1^2 - T_2 ^3)$.
    Choose $c \in k$ with $0 < \vert c \vert < 1$ and cover $X$ by
    \begin{align*}
     U_c &= \{ (x_1, x_2) \in X \mid \Vert x_1 \Vert \leq \vert c \vert ^ 2 \} \quad \text{and}\\
     V_c &= \{ (x_1, x_2) \in X \mid \Vert x_1 \Vert \geq \vert c \vert ^ 2 \} .
    \end{align*}
    Then $\mathcal{U}_c = \{ U_c, V_c \}$ is an admissible covering.
    The sheaf $\mathcal{O}_X (<1)$ on $X$ is defined by
    \begin{equation*}
     \mathcal{O}_X (<1) (U) := \{ f \in \mathcal{O}_X (U) \mid \Vert f \Vert < 1 \} .
    \end{equation*}
    The first \v{C}ech cohomology group of $X$ with respect to this covering with coefficients in $\mathcal{O}_X (<1)$ is:
    \begin{equation*}
     \check{H}^1 ( \{ U_c , V_c \} , \mathcal{O}_X (<1) ) \cong k ^ \circ / c k ^\circ . 
    \end{equation*}
    We have injections
    \begin{equation*}
       H^1 ( \mathcal{U}_c, \mathcal{O}_X (<1)) \hookrightarrow \check{H}^1 (X , \mathcal{O}_X (<1)) \cong H^1 (X , \mathcal{O}_X (<1)).
    \end{equation*}
    In particular, $H^1 (X , \mathcal{O}_X (<1))$ is non-trivial.
    
    Define a sheaf $\mathcal{G}_0$ on $X$ by
    \begin{align*}
     \mathcal{G}_0 (U) := \bigg\{ &f \in \mathcal{O}_{X\times \mathbb{B}^1} (U \times \mathbb{B}^1) = \mathcal{O}_X (U) \langle T \rangle
     \, \bigg| \, 
     f = 1+ \sum _{i=1}^\infty f_iT^i \\
     &\text{ with } \vert f_i(u) \vert < 1 \ \forall i \geq 1 \ \forall u \in U \bigg\} \quad \text{for $U \subset X$ admissible.}
    \end{align*}
    By \cite[Satz 2]{Gerritzen77} there is a decomposition
    \begin{equation*}
     \operatorname{Pic} (X \times \mathbb{B}^1) \cong \operatorname{Pic} (X) \oplus H^1 (X, \mathcal{G}_0).
    \end{equation*}
    If $\operatorname{char} k = 0$, then
    \begin{equation*}
     H^1 (X , \mathcal{O}_X (<1)) = 0 \enspace \Leftrightarrow \enspace  H^1 (X , \mathcal{G}_0) = 0
    \end{equation*}
    by \cite[Satz 4]{Gerritzen77} and hence if $\operatorname{char} k = 0$, then
    \begin{equation*}
     \operatorname{Pic} \big( \operatorname{Sp} k \langle T_1 , T_2 \rangle / (T_1^2 - T_2 ^3) \times \mathbb{B}^1 \big) \ncong 
     \operatorname{Pic} \big( \operatorname{Sp} k \langle T_1 , T_2 \rangle / (T_1^2 - T_2 ^3) \big) .
    \end{equation*}
   \end{Ex}
   \begin{Rem}[Concerning the notation]
    Gerritzen's notation differs from ours. Our sheaf $\mathcal{G}_0$ is his sheaf $\mathcal{G}_1$ whereas he denotes the
    sheaf $\mathcal{O}_X^\ast$ by $\mathcal{G}_0$. The sheaf that we call $\mathcal{O}_X (<1)$ is called $\check{\mathcal{H}}_X$
    in Gerritzen's article. We chose a notation similar to that of van der Put~\cite{vdPut82}.
   \end{Rem}

 \subsection{$\mathbb{B}^1$-invariance}
 This subsection collects known theorems about when the Picard group of a space is or is not $\mathbb{B}^ 1$-invariant.
 It contains no new material. 
 
 \begin{Thm}[Gerritzen]\label{ThmGerritzenPic}\index{$\mathbb{B}^1$-invariance}
  Let $k$ be a complete discretely valued field. Let $A$ be a reduced $k$-affinoid algebra with regular reduction $\tilde{A}$.
  Assume $\Vert A \Vert = \vert k \vert$. Then the Picard group is $\mathbb{B}^1$-invariant on $\operatorname{Sp} (A)$:
  \begin{equation*}
   \operatorname{Pic} (A) \overset{\star}{\cong} \operatorname{Pic} ( \tilde{A} ) \overset{\diamond}{\cong} 
   \operatorname{Pic} ( \widetilde{A \langle T \rangle}) \overset{\star}{\cong} \operatorname{Pic} (A \langle T \rangle ).
  \end{equation*}
 \end{Thm}
 \begin{proof}
  By Gerritzen's Theorem~\ref{ThmGerritzenPicc}, for each affinoid algebra $B$ with $\Vert B \Vert = \vert k \vert$ and 
  whose ring of integers $B^\circ$ is regular, there is an isomorphism
  \begin{equation*}
  \operatorname{Pic} (B) \cong \operatorname{Pic} ( \tilde{B} ).
 \end{equation*}
 This applies in particular when $\tilde{B}$ is regular and proves the isomorphisms $\star$ (keeping in mind that 
 $\Vert A \Vert = \Vert A \langle T \rangle \Vert$).
 The isomorphism $\diamond$ follows from the following theorem by 
 Bass and Murthy.
 \end{proof}
   
 \begin{Thm}[{Bass--Murthy \cite{BassMurthy67}}]\label{BassMurthy}
  Let $X$ be a separated reduced scheme over $k$ and assume that it has at most normal crossings singularities.
  Then
  \begin{equation*}
   \operatorname{Pic} (X) \cong \operatorname{Pic} (X \times \mathbb{A}^ 1).
  \end{equation*}
 \end{Thm}
 There is a particularly clearly presented proof of this theorem by Hartl--L\"utkebohmert \cite[Proposition~2.2 (1)]{HartlLutke00}.
 
 If $X$ is affine and onedimensional, there is an explicit decomposition of $\operatorname{Pic} ( X \times \mathbb{A}^ 1 )$:
 \begin{Thm}[Endo {\cite[Theorem~4.3]{Endo63}}, Bass--Murthy {\cite[Proposition~7.12]{BassMurthy67}}]\label{ThmBassMurthy}
  Let $R$ be a reduced ring of dimension $\leq 1$ and assume that its integral closure $\bar{R}$ is finitely generated as an $R$-module.
  Let $\mathfrak{b}$ be the conductor of $R$ in $\bar{R}$. Then
  \begin{equation*}
   \operatorname{Pic} (R [T]) \cong \operatorname{Pic} (R) \oplus G_0
  \end{equation*}
  where
  \begin{equation*}
   G_0 = 0 \enspace \Leftrightarrow \enspace \mathfrak{b} = \sqrt{\mathfrak{b}},
  \end{equation*}
  that is, if the conductor is a radical ideal.
 \end{Thm}
  
 A corresponding rigid analytic decomposition was given by Gerritzen:
 \begin{Thm}[Gerritzen {\cite[§3, Satz 2]{Gerritzen77}}]\index{$\mathbb{B}^1$-invariance}
  Let $X = \operatorname{Sp} A$ be an affinoid variety. Then
  \begin{equation*}
   \operatorname{Pic} ( X \times \mathbb{B}^1 ) \cong \operatorname{Pic} (X) \oplus H^1 (X, \mathcal{G}_0)
  \end{equation*}
  where the sheaf $\mathcal{G}_0$ on $X$ is defined as follows. For $U \subset X$ admissible,
  \begin{align*}
   \mathcal{G}_0 (U) := \bigg\{ &f \in \mathcal{O}_{X\times \mathbb{B}^1} (U \times \mathbb{B}^1) = \mathcal{O}_X (U) \langle T \rangle
   \, \bigg| \\
   &f = 1+ \sum _{i=1}^\infty f_iT^i 
   \text{ with } \vert f_i(u) \vert < 1 \ \forall i \geq 1 \ \forall u \in U \bigg\} .
  \end{align*}
 \end{Thm}
 Let $\iota _1 \colon X \rightarrow X \times \mathbb{B}^ 1, x \mapsto (x,1)$
 be the 1-section. The sheaf $\mathcal{G}_0$ on $X$ is the subsheaf of $\iota _1^\ast \mathcal{O}_{X \times \mathbb{B}^1}$
 consisting of the normalised invertible elements.
 
 \begin{Rem}
  As before, the sheaf $\mathcal{G}_0$ is called $\mathcal{G}_1$ in Gerritzen's article.
 \end{Rem}
 
   The following example occured first in Gerritzen~\cite[Beispiel 1, pp.~36f]{Gerritzen77} or 
   Bartenwerfer~\cite[p.~1]{Bartenwerfer78}.
   Bartenwerfer attributes it to Bosch.
   \begin{Ex}[Bad reduction]\label{Exbadreduction}
    $\mathbb{B}^1$-invariance can fail for a smooth affinoid variety whose reduction has a cusp singularity.
    Assume that $k$ is discretely valued and choose $\pi \in k$ with $0 < \vert \pi \vert < 1$. Let 
    \begin{equation*}
     A = k \langle T_1 , T_2 \rangle / (T_1 ^2 - T_2 (T_2 - \pi )(T_2 - 2 \pi )).
    \end{equation*}
    The reduction $\tilde{A}$ is an algebraic $\tilde{k}$-variety with a cusp singularity:
    \begin{equation*}
     \tilde{A} = \tilde{k} [ T_1, T_2 ] / (T_1^2 - T_2^3).
    \end{equation*}
    Now the ring of integers
    \begin{equation*}
     A^\circ = k^\circ \langle T_1 , T_2 \rangle / (T_1 ^2 - T_2 (T_2 - \pi )(T_2 - 2 \pi ))
    \end{equation*}
    is regular.
    By Gerritzen's Theorem~\ref{ThmGerritzenPicc},
    \begin{align*}
     \operatorname{Pic} (A) &\cong \operatorname{Pic} ( \tilde{A} ) \quad \text{and} \\
     \operatorname{Pic} ( \widetilde{A \langle T \rangle}) &\cong \operatorname{Pic} (A \langle T \rangle ).
    \end{align*}
    A theorem by Bass and Murthy \cite[Theorem~8.1]{BassMurthy67}
    entails that 
    \begin{equation*}
     \operatorname{Pic} ( \tilde{A} ) \cong \operatorname{Pic} ( \tilde{A} [T]) \enspace
     \Leftrightarrow \enspace \mathfrak{b} = \sqrt{\mathfrak{b}}
    \end{equation*}
    where $\mathfrak{b}$ denotes the conductor of the integral closure of $A$ in the quotient field $\operatorname{Quot} A$.
    In our case the conductor is
    \begin{equation*}
     \mathfrak{b} = \bigg( \bigg(\frac{T_1}{T_2} \bigg) ^2 \bigg) , 
    \end{equation*}
    clearly not a radical. Hence,
    \begin{equation*}
     \operatorname{Pic} ( \tilde{A} ) \ncong \operatorname{Pic} ( \tilde{A} [T]) 
    \end{equation*} 
    and, as $\widetilde{A \langle T \rangle} \cong \tilde{A} [T]$, this implies
    \begin{equation*}
     \operatorname{Pic} (A) \ncong \operatorname{Pic} ( A \langle T \rangle ).
    \end{equation*}
   \end{Ex}  
   \begin{Rem}
    \begin{enumerate}
     \item The same counterexample works if $k$ is algebraically closed instead of discretely valued \cite[Remark 3.24]{vdPut82}.
     \item We see that the choice of analytic reduction matters. 
         If $\Vert A \Vert = \vert k \vert$, then the canonical reduction knows all the vector bundles of a smooth affinoid variety.
         Let $k = \mathbb{C} ((T))$ with $\vert T \vert = p^{-1}$ and $p \geq 5$ prime. Take $X$ as in Example \ref{Exbadreduction}. 
         Its canonical formal model
         \begin{equation*}
          \mathcal{X} = \mathbb{C} [[T]] \langle \xi _1, \xi _2 \rangle / \big( \xi _1 ^2 - \xi _2 ( \xi _2 - T)( \xi _2 - 2T) \big)
         \end{equation*}
         is affine and contains all information about the vector bundles of $X$.
         Similarly, the canonical formal model for $Y := X \times \mathbb{B}^1$  is 
         \begin{equation*}
          \mathcal{Y} = \mathbb{C} [[T]] \langle \xi _1, \xi _2 , \xi \rangle / \big( \xi _1 ^2 - \xi _2 ( \xi _2 - T)( \xi _2 - 2T) \big) .
         \end{equation*}
         Blowing up three times in $(T, \xi _1 , \xi _2 )$ yields a resolution of singularities,
         giving rise to regular formal schemes $\tilde{\mathcal{X}}$, $\tilde{\mathcal{Y}}$ which are models for
         $X$, respectively for $Y$, but no longer affine.
         Look at the special fibres:
         $\tilde{\mathcal{X}}_\sigma$ is the preimage of $\mathbb{C}[\xi _1, \xi _2] / (\xi _1^2 - \xi _2^3)$ 
         under the triple blowup of $\mathbb{C}[\xi _1, \xi _2]$ in the point 
         $( \xi _1 = 0, \xi _2 = 0)$ and
         $\tilde{\mathcal{Y}}_\sigma$ is the preimage of $\mathbb{C}[\xi _1, \xi _2, \xi ] / (\xi _1^2 - \xi _2^3)$ 
         under the triple blowup of $\mathbb{C}[\xi _1, \xi _2, \xi ]$ in the divisor $( \xi _1 = 0, \xi _2 = 0)$. 
         We are in the situation where
         $\tilde{\mathcal{X}}_\sigma$ is regular and $\tilde{\mathcal{Y}}_\sigma = \tilde{\mathcal{X}}_\sigma \times \mathbb{A}^1$.
         They have the same isomorphism classes of line bundles by Theorem~\ref{BassMurthy} (Bass--Murthy).
         Hence, some formal models contain the information about their generic fibres' vector bundles
         while other formal models do not.
    \end{enumerate}
   \end{Rem}
 
 For $k$ algebraically closed there are results by van der Put and by Bartenwerfer:
 \begin{Thm}[van der Put, Bartenwerfer]\label{Putdecompthm}\index{$\mathbb{B}^1$-invariance}
  Let $k$ be a non-archimedean valued, complete and algebraically closed field.
  Let $A$ be an affinoid $k$-algebra. 
  \begin{enumerate}
   \item\label{Putdecomp} There exists a sheaf $\mathcal{G}$ on $\operatorname{Sp} (A)$ such that
    \begin{equation*}
   \operatorname{Pic} (A \langle T \rangle ) \cong \operatorname{Pic} (A) \enspace \Leftrightarrow \enspace 
    H^i ( \operatorname{Sp} (A) , \mathcal{G}) = 0.
  \end{equation*}
   \item\label{BartO1} If $k$ is algebraically closed, $\operatorname{char} \tilde{k} = 0$ and
       $\operatorname{Sp} (A)$ is a reduced affinoid $k$-variety of good reduction, then 
       \begin{equation*}
        H^i ( \operatorname{Sp} (A) , \mathcal{G} ) = 0 .
       \end{equation*}
   \item\label{Putcounterex} There exists an affinoid $k$-algebra $A$ such that
        $H^1 ( \operatorname{Sp} (A) , \mathcal{G}) \neq 0$.
  \end{enumerate}
 \end{Thm}
 \begin{proof}
  \begin{itemize}
   \item[\ref{Putdecomp})] follows from \cite[Proposition~3.32 (3)]{vdPut82}. The number $s$ in van der Put's theorem is the number
      of holes in the disc $D$. In our case, $s=0$.
   \item[\ref{BartO1})] By \cite[Proposition~3.32 (4)]{vdPut82}, we know that
      $H^i ( \operatorname{Sp} (A) , \mathcal{G} ) = 0$ if and only if 
      \begin{equation*}
       H^i \big( \operatorname{Sp} (A) , \mathcal{O}_{\operatorname{Sp} (A)} (<1) \big) \overset{\star}{=} 0.
      \end{equation*}
      As $\operatorname{Sp} (A)$ is a reduced affinoid of good reduction, ($\star$) holds by~\cite[Theorem~2']{Bartenwerfer78}. 
      The proof is at the very end of Bartenwerfer's article.
   \item[\ref{Putcounterex})] Take a variety of bad reduction like in Example \ref{Exbadreduction}.
     Then $H^1 ( \operatorname{Sp} (A) , \mathcal{G}) \neq 0$ \cite[Remark 3.24]{vdPut82}. \qedhere
  \end{itemize}
 \end{proof}
  
 A special case is
 \begin{Prop}[{\cite[Corollary~3.29]{vdPut82}}]
  Let $\operatorname{Sp} A$ be a $k$-affinoid space such that for all $r \in \mathbb{R}$ and all constant sheaves $G$ we have
  \begin{align*}
    H^i \big( \operatorname{Sp} (A) , \mathcal{O}_{\operatorname{Sp} (A)} (<r) \big) &= 0 \enspace \text{for} \enspace i \neq 0 
      \enspace \text{and}\\
    H^i \big( \operatorname{Sp} (A) , G \big) &= 0 \enspace \text{for} \enspace i \neq 0.
  \end{align*}
  Then
  \begin{equation*}
   H^i \big( \operatorname{Sp} (A\langle T \rangle ) , \mathcal{O}^\ast \big) \cong 
   H^i \big( \operatorname{Sp} (A) , \mathcal{O}^\ast \big) \enspace \text{for} \enspace i \neq 0.
  \end{equation*}
  In particular,
  \begin{equation*}
   \operatorname{Pic} (A \langle T \rangle ) \cong \operatorname{Pic} (A).
  \end{equation*}
 \end{Prop}
 
 Having an $\A ^1$-invariant Picard group is a property of schemes which is Zariski-local.
 This means that to prove $\A ^1$-invariance of the Picard group of a schame, it is enough to verify it on
 Zariski-open subschemes. In particular, $\A ^1$-invariance of Pic for smooth (or normal) affine varieties
 implies $\A ^1$-invariance of Pic for all smooth (or normal) algebraic varieties.
 The preceding proposition shows that the analogue for affinoid varieties does not hold:
 Whether or not the Picard group of an affinoid variety is $\B ^1$-invariant depends on
 the canonical reduction of the variety. As smoothness (or normality) of the canonical reduction
 is a property which is not local with respect to the G-topology,
 also $\B^1$-invariance of the Picard group is a property which is not local with respect to the G-topology.
   
 By Theorem~\ref{ThmClass}, a homotopy classification of vector bundles follows directly from homotopy invariance, under
 two conditions: Firstly, the Grothendieck topology that we work with needs to be completely decomposable.
 Secondly, if homotopy invariance holds for rigid analytic varieties with some property (A), then we need that
 every vector bundle over an object with property (A) has a local trivialisation by subsets that also have the property (A).
 In other words, the conditions under which homotopy invariance holds should be local on the base space.
 Therefore, $\B^1$ is not a suitable interval object for our approach.

\section{Homotopy theory for rigid analytic varieties}\label{sechotheo}
 Joseph Ayoub was the first to construct and investigate motivic homotopy theory for rigid analytic varieties~\cite{AyoubB1}.
 His homotopy category of rigid analytic varieties is Morel--Voevodsky's homotopy category 
 of a site with interval \cite{MV} where the site is the category of rigid analytic varieties
 equipped with the Grothendieck topology coming from the G-topology  
 and with the unit ball $\mathbb{B}^1$ as an interval object.
 Ayoub used stable versions of this theory to construct motives for rigid analytic varieties.
 Further results in this line were obtained by Alberto Vezzani \cite{Vezzani14a,Vezzani14b,Vezzani15}.
 
 We will use a more flexible definition
 allowing three parameters: 
 \begin{itemize}
  \item A subcategory $\mathcal{R}$ of the category of rigid analytic varieties,
  \item a Grothendieck topology $\mathcal{T}$ on $\mathcal{R}$, usually the one coming from the G-topology or the Nisnevich topology, and
  \item an interval object $I$, usually $I = \mathbb{B}^ 1$ or $I = \mathbb{A}^ 1_\mathrm{rig}$.
 \end{itemize}
 For now, we fix none of these parameters. We define these notions in the next subsection.
 
 \subsection{$\mathbb{A}^1_\mathrm{rig}$-Homotopy theory}
 A \emph{model category} is always meant in the sense of Hirschhorn \cite[7.1, 13.11]{Hirschhorn}, 
 that is, a closed Quillen model category containing all small limits and colimits and with functorial factorisation.
 For \emph{simplicial model categories} we refer to Goerss--Jardine \cite{GoerssJardine}.
 We denote by $\sSet$ the category of simplicial sets endowed with the classical model structure
 where cofibrations are levelwise injections, weak equivalences are those morphisms whose geometric
 realisation is a weak homotopy equivalence and fibrations are Kan fibrations \cite[chapter I]{GoerssJardine}.
 
 \begin{Deff}[site, cd structure, {\cite[Definition 2.1]{Voevodsky10}}]\label{defcd}\label{site}
  \begin{enumerate}
   \item A \emph{site} $(\mathcal{C}, \mathcal{T})$ is a category $\mathcal{C}$ equipped with a Grothendieck topology $\mathcal{T}$. 
       For the definition of Grothendieck topology, we refer to
       Mac Lane--Moerdijk \cite[chapter 3]{MacLaneMoerdijk94}.
   \item Let $\mathcal{C}$ be a category with an initial object $\emptyset$. A \emph{cd structure}\index{cd structure}
       on $\mathcal{C}$ is a collection $P$ of commutative squares
       \begin{equation}\label{Quadrat}   
         \begin{xy}\xymatrix{
          B \ar[r]\ar[d] & Y \ar[d]^p\\
          A \ar[r]_e &X
         }\end{xy}
       \end{equation}
       such that if $Q \in P$ and $Q'$ is isomorphic to $Q$, then $Q' \in P$.
       The Grothendieck topology $t_P$ generated by the cd structure $P$ is the coarsest topology on $\mathcal{C}$ such that
       \begin{itemize}
        \item the empty sieve covers the initial object $\emptyset$ and
        \item for every $Q \in P$ as in \eqref{Quadrat}, the sieve generated by the morphisms $p$ and $e$ is a covering sieve.
       \end{itemize}
   \item A Grothendieck topology which is generated by a cd structure is called 
       \emph{completely decomposable}\index{completely decomposable topology}.
  \end{enumerate}
 \end{Deff}
 
 \begin{Exx}\label{Exsites}
  \begin{enumerate}
   \item\label{GSitus} Let $X$ be a rigid analytic variety with G-topology $\mathcal{T}$. 
      The category $\underline{X}$ has as objects the admissible open subsets of $X$ and
      inclusions as morphisms. With admissible coverings as covering sieves, it becomes a site, also denoted by $\underline{X}$.
      
      The Grothendieck topology of $\underline{X}$ is completely decomposable if and only
      if $X$ is quasicompact.
   \item\label{ZarSitus} Let $X = \operatorname{Sp} A$ be an affinoid variety. The site $X_\mathrm{Zar}$ has as objects Zariski open
      subsets of $X$, as morphisms inclusions and as covering sieves the coverings by Zariski open subsets.
      
      The Grothendieck topology of the Zariski site $X_\mathrm{Zar}$ of an affinoid variety $X$ is
      generated by squares
      \begin{equation*}
       \begin{xy}\xymatrix{
        U \cap V \ar[r]\ar[d] & V \ar[d]^\cap \\
        U \ar[r]_\subset &X
       }\end{xy}
      \end{equation*}
      where $U \to X$ and $V \to X$ are the inclusions of Zariski open subsets of $X$ that jointly cover the affinoid variety $X$.
   \item\label{EtSitus} Let $X$ be a separated rigid analytic variety. The {\'e}tale site $X_\mathrm{et}$ consists of the following.
      \begin{itemize}
       \item The objects are {\'e}tale morphisms $Y \to X$ from separated rigid analytic varieties $Y$ to $X$.
       \item The morphisms are commutative triangles
          \begin{equation*}
            \begin{xy}\xymatrix{
             Y \ar[d]\ar[r] & X \\
             Z \ar[ru] &
            }\end{xy}
           \end{equation*}
       \item The covering sieves are families of {\'e}tale morphisms whose images form admissible coverings.
      \end{itemize}
      The Grothendieck topology of $X_\mathrm{et}$ is usually not completely decomposable.
  \end{enumerate}
 \end{Exx}

 \begin{Def}[representable interval object, {\cite[4.1.1]{AsokHoyoisWendt15}}]\index{interval object}\label{definterval}
   A \emph{(representable) interval object} on a small category $\mathcal{C}$ is a quadruple $(I, m, \iota _0 , \iota _1)$
   consisting of a representable presheaf of (simplicial) sets $I$ on $\mathcal{C}$, 
   a morphism (``multiplication'') $m\colon I \times I \to I$ and two morphisms (``endpoints'') $\iota _0 , \iota _1 \colon \ast \to I$ 
   such that the following hold:
   \begin{enumerate}
    \item For every $X \in \operatorname{Ob} ( \mathcal{C})$, the presheaf 
        $\operatorname{Hom}_{\mathcal{C}} ( \ \cdot \ ,X) \times I$ is representable.
    \item Let $p\colon I \to \ast$. Then the following morphisms $I \to I$ coincide:
       \begin{align*}
        m ( \iota _0 \times \id ) &= m ( \id \times \iota _0) = \iota _0 p \enspace \text{and}\\
        m ( \iota _1 \times \id ) &= m ( \id \times \iota _1) = \id .
       \end{align*}
    \item The morphism $\iota _0 \coprod \iota _1 \colon \ast \coprod \ast \to I$ is a monomorphism.
   \end{enumerate}
   We will sometimes drop the word ``representable'' and we will suppress the morphisms $m, \iota _0, \iota _1$ in the notation.
 \end{Def}

 \begin{Prop}\label{rudel}
  Let $\mathcal{C}$ be a small category equivalent to $\rigVar$. Let $\mathcal{R} \subset \mathcal{C}$ be a subcategory.
  Let $\mathit{sPSh} ( \mathcal{R} )$ a small category equivalent to the category of simplicial presheaves on $\mathcal{R}$.
  Let $I$ be a representable interval object and $\mathcal{T}$ a Grothendieck topology
  on $\mathcal{R}$. Then there are simplicial model structures on $\mathit{sPSh} ( \mathcal{R} )$ as follows.
  \begin{enumerate}
   \item\label{teekanne} The \emph{injective model structure} $\mathcal{M} ( \mathcal{R})$\index{$\mathcal{M} ( \mathcal{R})$} where
      \begin{itemize}
       \item the cofibrations are the pointwise monomorphisms,
       \item the weak equivalences are sectionwise weak equivalences, i.\,e., two simplicial presheaves
           $\mathcal{F}, \mathcal{G} \in \mathit{sPSh} ( \mathcal{R} )$ are weakly equivalent if and only if
           \begin{equation*}
            \mathcal{F} ( X) \simeq \mathcal{G} (X) \quad \text{in } \sSet
           \end{equation*}
           for all $X \in \operatorname{Ob} (\mathcal{R})$, and
       \item the fibrations are defined via the right lifting property.
	  \end{itemize} 
   \item\label{kaffeekanne} The model structure 
      $\mathcal{M}_{I, \mathcal{T}} ( \mathcal{R})$\index{$\mathcal{M}_{I, \mathcal{T}} ( \mathcal{R})$}  obtained by
      left Bousfield localisation of $\mathcal{M} ( \mathcal{R})$ at the set
      \begin{equation*}
       \{ I \times \mathcal{F} \rightarrow \mathcal{F} \mid \mathcal{F} \in \mathit{sPSh} ( \mathcal{R} ) \} 
       \cup \operatorname{Cov} ( \mathcal{T} )
      \end{equation*}
      where $\ast$ denotes the terminal object and where $\operatorname{Cov} ( \mathcal{T} )$ is the set of all 
      covering sieves of $\mathcal{T}$.
  \end{enumerate}
 \end{Prop}
 \begin{proof}
  The existence of these model categories was shown by Cisinski~\cite{Cisinski02}.
 \end{proof}

 \begin{Def}[$\mathcal{H}_{I,  \mathcal{T} } ( \mathcal{R})$]\index{$\mathcal{H}_{I,  \mathcal{T} } ( \mathcal{R})$}
  The homotopy category of $\mathcal{M}_{I, \mathcal{T} } (\mathcal{R})$ is denoted by $\mathcal{H}_{I,  \mathcal{T} } ( \mathcal{R})$.
  If the Grothendieck topology $\mathcal{T}$ is generated by a cd structure $\tau$ (see Definition \ref{defcd}),
  we also denote the resulting categories by $\mathcal{M}_{I, \tau } (\mathcal{R})$, respectively by
  $\mathcal{H}_{I,  \tau } ( \mathcal{R})$.
 \end{Def}
 
 We view the category $\mathcal{R} \subset \rigVar$ 
 as a subcategory of $\mathcal{M}_{I, \mathcal{T}} ( \mathcal{R})$ via the Yoneda embedding.
 
 \begin{Def}
  Let $\mathcal{R}$ be a small category that has an initial object $\emptyset$ as in Proposition~\ref{rudel}.
  Let $I$ be a representable interval object, $\tau$ a cd structure on $\sPSh ( \mathcal{R} )$, let
  $\mathcal{M}_{I, \tau} ( \mathcal{R})$ be the model category constructed in Proposition~\ref{rudel} and
  $\mathcal{H}_{I, \tau} ( \mathcal{R})$ the resulting homotopy category.
  Let $\mathcal{E}, \mathcal{F} \in \operatorname{Ob} ( \sPSh ( \mathcal{R} ))$. The \emph{set of homotopy classes} 
  from $\mathcal{E}$ to $\mathcal{F}$ is
  \begin{equation*}
   [\mathcal{E}, \mathcal{F}]_{I, \tau} := \mathit{Map} _{\mathcal{H}_{I, \tau } ( \mathcal{R})} (\mathcal{E}, \mathcal{F}) 
   := \Hom _{\mathcal{M}_{I, \tau } ( \mathcal{R})} (\mathcal{E}, R \mathcal{F}) / \text{homotopy}
  \end{equation*}
  where $R$ is a fibrant replacement functor.
  As all objects of $\mathcal{M}_{I, \tau } ( \mathcal{R})$ are cofibrant, we do not need to replace $\mathcal{E}$
  cofibrantly.
 \end{Def}
  
 \begin{Lemma}\label{Manta}
  As in Proposition~\ref{rudel}, let $\mathcal{R}$ be a small category equivalent to a subcategory of $\rigVar$ such that
  $\mathbb{B}^1$ and $\mathbb{A}^1_\mathrm{rig} \in \mathit{sPSh} ( \mathcal{R} )$ are representable interval objects.
  Let the Grothendieck topology $\mathcal{T}$ be equal to or finer than the Grothendieck topology coming
  from the G-topology.
  Then the rigid analytic affine line $\mathbb{A}^1_\mathrm{rig}$ is contractible 
  in $\mathcal{M}_{\mathbb{B}^1 , \mathcal{T}} ( \mathcal{R})$.
 \end{Lemma}
 \begin{proof}
  We may assume without loss of generality that $\mathcal{T} = G$. If $\mathcal{T} \neq G$, then 
  $\mathcal{M}_{\mathbb{B}^1 , \mathcal{T}} ( \mathcal{R})$ is obtained by further localisation of 
  $\mathcal{M}_{\mathbb{B}^1 , G} ( \mathcal{R})$. Localisation preserves weak equivalences, thus if
  $\mathbb{A}^1_\mathrm{rig}$ is contractible in $\mathcal{M}_{\mathbb{B}^1 , G} ( \mathcal{R})$,
  it is still so in $\mathcal{M}_{\mathbb{B}^1 , \mathcal{T}} ( \mathcal{R})$.
  
  Let $c \in k$ with $\vert c \vert > 1$ and set 
  \begin{equation*} 
   \mathbb{A} ^1_\mathrm{rig} = \operatorname{colim} \big( \Sp (k \langle \xi \rangle ) \hookrightarrow 
   \Sp ( k \langle c^{-1} \xi \rangle ) \hookrightarrow \Sp ( k \langle c^{-2} \xi \rangle ) \hookrightarrow \dotso \big)
  \end{equation*}
  as in Example \ref{DefA1rig}. The colimit of simplicial sheaves
  coincides with the simplicial sheaf represented by the rigid analytic variety $\mathbb{A} ^1_\mathrm{rig}$. 

  We need to show that for each rigid analytic variety $X$, the map $X \times \mathbb{A} ^1_\mathrm{rig} \rightarrow X$
  is a weak equivalence in $\mathcal{M}_{\mathbb{B}^1} ( \mathcal{R} )$, i.\,e.~that the projection
  \begin{equation*}
   X \times \mathbb{A} ^1_\mathrm{rig} \rightarrow X
  \end{equation*}
  is a $\mathbb{B} ^1$-local equivalence in $\mathcal{M} ( \mathcal{R} )$.
  That is, if $Y$ is any $\mathbb{B} ^1$-local object, the induced map of mapping spaces
  \begin{equation*}
   \mathit{Map} (X,Y) \rightarrow \mathit{Map} (X \times \mathbb{A} ^1_\mathrm{rig} , Y)
  \end{equation*}
  is a weak equivalence of simplicial sets.

  Let $Y$ be a $\mathbb{B}^1$-local object. 
  The monomorphisms $\Sp ( k \langle c^{-i} \xi \rangle ) \hookrightarrow \Sp ( k \langle c^{-(i+1)} \xi \rangle )$ 
  are cofibrations in $\mathcal{M} ( \mathcal{R} )$. As $Y$ is fibrant in $\mathcal{M} ( \mathcal{R} )$, the maps
  \begin{equation*}
   \mathit{Map} (X \times \Sp k \langle c^{-(i+1)} \xi \rangle , Y) \rightarrow 
   \mathit{Map} (X \times \Sp k \langle c^{-i} \xi \rangle , Y)
  \end{equation*}
  are fibrations of simplicial sets by axiom SM7 \cite[Chapter II, Proposition~3.2]{GoerssJardine}.
  The isomorphisms
  \begin{align*}
    k \langle c^{-i} \xi \rangle &\rightarrow  k \langle \xi \rangle \\
    \xi &\mapsto c^i \xi \\
    a & \mapsto a \quad \text{for } a \in k
  \end{align*}
  give rise to isomorphisms
  \begin{equation*}
   \Sp k \langle \xi \rangle \rightarrow \Sp k \langle c^{-i} \xi \rangle
  \end{equation*}
  which in particular are weak equivalences. As the morphism
  \begin{equation*}
   \Sp k \langle \xi \rangle \rightarrow \ast
  \end{equation*}
  is a weak equivalence in $\mathcal{M} _{\mathbb{B}^1} ( \mathcal{R} )$ by definition, the two-out-of-three axiom gives that
  \begin{equation*}
   \Sp k \langle c^{-i} \xi \rangle \rightarrow \ast
  \end{equation*}
  is a weak equivalence in $\mathcal{M} _{\mathbb{B}^1} ( \mathcal{R} )$. Thus we get weak equivalences of simplicial sets
  \begin{equation*}
   f_i\colon \mathit{Map} (X,Y) \rightarrow \mathit{Map} ( X \times \Sp k \langle c^{-i} \xi \rangle , Y)
  \end{equation*}
  for all $i$.
  Similarly, the inclusions
  \begin{align*}
    k \langle c^{-(i+1)} \xi \rangle &\rightarrow  k \langle c^{-i} \xi \rangle \\
    \xi &\mapsto \xi \\
    a & \mapsto a \quad \text{for } a \in k
  \end{align*}
  give rise to cofibrations 
  \begin{equation*}
   \Sp ( k \langle c^{-i} \xi \rangle ) \rightarrow \Sp ( k \langle c^{-(i+1)} \xi \rangle )
  \end{equation*}
  in $\mathcal{M} _{\mathbb{B}^1} ( \mathcal{R} )$ and the maps
  \begin{equation*}
   j_n\colon \mathit{Map} (X \times \Sp k \langle c^{-(j+1)} \xi \rangle , Y ) \rightarrow 
   \mathit{Map} (X \times \Sp k \langle c^{-j} \xi \rangle , Y )
  \end{equation*}
  are fibrations in $\sSet$.
  We put this together to get a sequence of maps
  \begin{equation*}
   \begin{xy} \xymatrix{
    & & \mathit{Map} (X , Y) \ar[d]^{f_0} \ar[dl]^{f_1} \ar[dll]_{f_2} \\
    \dotso \ar[r] & \mathit{Map} (X \times \Sp k \langle c^{-1} \xi \rangle , Y ) \ar[r]_{j_0} 
      & \mathit{Map} (X \times \Sp k \langle \xi \rangle , Y ) .   
   } \end{xy}
  \end{equation*}
  The spaces $X \times \Sp k \langle c^{-i} \xi \rangle$ are cofibrant and 
  $Y$ is fibrant in $\mathcal{M} _{\mathbb{B}^1} ( \mathcal{R} )$, hence 
  all appearing mapping spaces are fibrant in $\sSet$~\cite[Chapter~II, Proposition~3.2]{GoerssJardine}. 
  All horizontal maps $j_n$ are fibrations
  and all vertical maps $f_n$ for $n \geq 0$ are weak equivalences by Axiom SM7 \cite[Proposition~3.2]{GoerssJardine}.
  By \cite[Proposition~15.10.12(2)]{Hirschhorn}, the limit map
  \begin{equation*}
   \operatorname{lim} f_n \colon \mathit{Map} (X,Y) \rightarrow \mathit{Map} ( X \times \mathbb{A} ^1 _\mathrm{rig} , Y )  
  \end{equation*}
  is a weak equivalence of simplicial sets. 
  Thus, the map
  \begin{equation*}
   X \times \mathbb{A} ^1_\mathrm{rig} \rightarrow X
  \end{equation*}
  is a $\mathbb{B} ^1$-local equivalence and $X$ and $X \times \mathbb{A} ^1_\mathrm{rig}$
  are isomorphic in the homotopy category $\mathcal{H} _ {\mathbb{B}^1} ( \mathcal{R} )$.
 \end{proof}
 
 \begin{Cor}\label{finerstructure}
  For each $\mathcal{R}$ as above, the model category $\mathcal{M} _{\mathbb{B} ^1} ( \mathcal{R} )$ 
  can be obtained by left Bousfield localisation
  of $\mathcal{M} _{\mathbb{A} ^1_\mathrm{rig}} ( \mathcal{R} )$ at the set 
  $\{ \mathbb{B}^ 1 \times \mathcal{F} \rightarrow \mathcal{F} \mid \mathcal{F} \in \mathit{sPSh} ( \mathcal{R} ) \}$.
 \end{Cor}
  
\subsection{Overview of the occurring theories}
 Denote by $\mathit{qStein}$ the category of smooth quasi-Stein varieties. 
 
 The occurring Grothendieck topologies are 
 \begin{itemize}
  \item the Grothendieck topology induced by the G-topology, denoted by $G$,
  \item the completely decomposable Grothendieck topology where only finite coverings with respect to the G-topology are allowed,
      denoted by $cdG$, and
  \item the Nisnevich topology, denoted by $\mathrm{Nis}$.
 \end{itemize}
 The Grothendieck topology induced by the G-topology is finer than its completely decomposable variant. They coincide
 on quasicompact rigid analytic varieties. The Nisnevich topology is finer than the 
 completely decomposable Grothendieck topology. 
 
 For $\mathcal{R} \in \{ \rigVar , \mathit{qStein} \}$ we have the following relations
 between the homotopy categories:
 \begin{equation*}
  \begin{xy}
   \xymatrix{
    \mathcal{H}_{\mathbb{A}^1_\mathrm{rig}, G} ( \mathcal{R} ) \ar[d]^{\text{loc}} & 
        \mathcal{H}_{\mathbb{A}^1_\mathrm{rig}, cdG} ( \mathcal{R} ) \ar[l]^{\text{loc}} \\
    \mathcal{H}_{\mathbb{B}^1, G} ( \mathcal{R} ) & \mathcal{H}_{\mathbb{B}^1, cdG} ( \mathcal{R} ). \ar[l]^{\text{loc}}
   }
  \end{xy}
 \end{equation*}
 Arrows labelled ``loc'' denote localisation of the model structure.
 
\subsection{Relation to Ayoub's theory}\label{RelAyoub}
 We recall Ayoub's construction \cite[1.3]{AyoubB1}.
 Ayoub starts with the category of presheaves from $\rigVar$ into a \emph{coefficient category} $\mathfrak{M}$.
 A coefficient category is a stable model category with nice properties, 
 cf.~\cite[D\'{e}finition~4.4.23]{Ayoub07b} or \cite[D\'{e}finition~1.2.31]{AyoubB1}.
 Examples include the category of complexes of abelian groups or categories of simplicial spectra, 
 endowed with the stable projective model strucure \cite[Exemple~4.4.24]{Ayoub07b}.
 
 Denote the category of presheaves from $\rigVar$ into the coefficient category $\mathfrak{M}$ by
 $\mathit{PSh} ( \rigVar , \mathfrak{M} )$. Ayoub endows it with the \emph{projective model structure}
 where a morphism of presheaves $\mathcal{F} \to \mathcal{G}$ is
 \begin{itemize}
  \item a weak equivalence if and only if for each
       $X \in \operatorname{Ob} ( \rigVar )$, it induces a weak equivalence $\mathcal{F} (X) \simeq \mathcal{G} (X)$
       in $\mathfrak{M}$,
  \item a fibration if and only if for each
       $X \in \operatorname{Ob} ( \rigVar )$, it induces a fibration $\mathcal{F} (X) \to \mathcal{G} (X)$
       in $\mathfrak{M}$ and
  \item a cofibration if and only if it satisfies the left lifting property with respect to all trivial fibrations.
 \end{itemize}
 Now Ayoub localises this model structure with respect to the Nisnevich topology and the unit ball $\mathbb{B}^1$.
 The homotopy category of the resulting stable model category is called the 
 \emph{unstable $\mathbb{B}^1$-homotopy category} \cite[1.3.1]{AyoubB1}.
 
 Let us compare this to our constructions:
 We start with a subcategory $\mathcal{R}$ of $\rigVar$ and presheaves from $\mathcal{R}$ to simplicial sets.
 The category of simplicial sets is not a coefficient category in Ayoub's sense.
 We first endow the resulting category $\sPSh ( \mathcal{R} )$ with the injective model structure 
 and then localise with respect to an interval object $I$ and a Grothendieck topology $\mathcal{T}$.
 \begin{itemize}
  \item\label{stabinj} Let us stabilise the simplicial model category which we get with
      $\mathcal{R} = \rigVar$, $I = \mathbb{B}^ 1$ and $\mathcal{T} = \mathrm{Nis}$
      in the simplicial direction. 
  \item Let us choose the category of $S^ 1$-spectra as a coefficient category in Ayoub's construction.
 \end{itemize}
 This gives two different model structures on $\mathit{PSh} ( \rigVar , \mathfrak{M} )$.
 They are Quillen equivalent: The unstable injective and the unstable projective model categories 
 on $\mathit{PSh} ( \rigVar )$ are Quillen equivalent via the identity (every projective fibration
 is an injective fibration and every injective cofibration is a projective cofibration).
 The model categories obtained from the injective and the projective model categories by left
 Bousfield localisation at the interval and the coverings (cf.~\ref{rudel}\ref{kaffeekanne}))
 are then Quillen equivalent by a theorem by Hirschhorn \cite[Theorem~3.3.20]{Hirschhorn}. 
 Finally, the stable localised injective
 and the stable localised projective model categories are Quillen equivalent by a result by Hovey \cite[Theorem~5.5]{Hovey01}.
  
 Ayoub also has a \emph{stable $\mathbb{B}^1$-homotopy category}~\cite[1.3.3]{AyoubB1}:
 As before, he starts with $\mathit{PSh} ( \rigVar , \mathfrak{M} )$, additionally requiring $\mathfrak{M}$
 to be symmetric monoidal and unital. Endowing $\mathit{PSh} ( \rigVar , \mathfrak{M} )$
 with the projective model structure as above yields a stable model category which is again symmetric monoidal and unital.
 Localising it with respect to the Nisnevich topology and the unit ball $\mathbb{B}^1$ and stabilising with respect to
 $\mathbb{P}^1_\mathrm{rig}$ yields the stable $\mathbb{B}^1$-homotopy category. 
 It is still further away from our model categories.

\section{Representability}
 We prove that isomorphism classes of line bundles over a smooth quasi-Stein rigid analytic variety 
 are represented in the $\mathbb{A}^1_\mathrm{rig}$-homotopy category.
 
 At the beginning of the section we review the known theorems that served as models for our theorem 
 and its proof: Steenrod's homotopy classification of continuous vector bundles over paracompact Hausdorff spaces,
 Grauert's homotopy classification of holomorphic vector bundles over complex Stein spaces
 and Fabien Morel's $\mathbb{A}^ 1$-homotopy classification of algebraic vector bundles over smooth affine varieties.
 After an account on classifying spaces we finally come to our main theorem and its proof.
 
 The author's thesis devotes a section to the discussion of the role of h-principles in the proofs \cite[Section 5.2]{Sigloch16}.
 As it is but a collection of known results, we left it out in this paper.
\subsection{Classification of vector bundles: The classical results}\sectionmark{The classical results}\label{secVBclass}
 The existence of a classifying space for vector bundles is a fundamental result in the studies of topological vector bundles and K-theory:
 \begin{Thm}[Steenrod]\label{ClassTopThm}
  Let $X$ be a paracompact Hausdorff space. There is a one-to-one correspondence
  \begin{equation*}
   [X, \mathrm{Gr}_{n, \mathbb{R}}] \overset{1:1}{\longleftrightarrow} \{ \text{real vector bundles of rank $n$ over } X \} / \cong
  \end{equation*} 
  between homotopy classes of maps from $X$ into the infinite Grassmannian and isomorphism classes of real vector bundles over $X$.
 \end{Thm}
 \begin{Def}[infinite Grassmannian]\label{DefGrn}\index{$\mathrm{Gr}_n$}\index{Grassmannian}
  \begin{enumerate}
   \item \emph{Topological:}
       Let $K \in \{ \R , \C \}$ and for $n,m \in \mathbb{N}$, $n \leq m$,
       $\operatorname{Gr}_{n,m,K} := \{ n\text{-dimensional subvectorspaces of }K^ m \}$
       the finite Grassmannians.\index{Grassmannian!finite}
       The \emph{infinite Grassmannian}\index{Grassmannian!infinite} $\operatorname{Gr}_{n,K}$ is the colimit 
       \begin{equation*}
        \operatorname{Gr}_{n,K} := \bigcup _{m \geq n} \operatorname{Gr}_{n,m,K}.
       \end{equation*}
       over the maps induced by inclusions 
       $K^m \overset{\sim}{\longrightarrow} K^m \oplus 0 \oplus \dotsb \oplus 0 \subset K^{m'}$ for $m < m'$.
       The topology given by the archimedean norm induces a topology
       on each of the $\operatorname{Gr}_{n,m}$ and the final topology on $\operatorname{Gr}_n$.
   \item \emph{Algebraic:} As an algebraic variety, the finite Grassmannian $\operatorname{Gr}_{n,m,K, \mathrm{alg}}$ 
       is defined as a quotient of $\operatorname{GL}_{m,K}$ 
       by the parabolic subgroup that stabilises the flag $0 \subset K^n \subset K^m$. 
       As such, it is smooth and projective \cite[1.8]{Humphreys75}. 
       Taking the colimit over the finite Grassmannians, we get an ind-variety,
       the algebraic infinite Grassmannian $\operatorname{Gr}_{n, \mathrm{alg}}$.
   \item \emph{Complex analytic:} Similarly, $\operatorname{Gr}_{n,m, \mathbb{C}}$ admits a structure of complex manifold, 
       cf.~\cite[4.12]{BrockertomDieck95} and we get
       $\operatorname{Gr}_{n,\mathbb{C}}$, an \emph{complex ind-manifold}.
   \item \emph{Rigid analytic:} If $K=k$ is non-archimedean, 
       the rigid analytic variety $\operatorname{Gr}_{n,m,k, \mathrm{rig}}$ is defined as
       the analytification of the algebraic variety $\operatorname{Gr}_{n,m,k, \mathrm{alg}}$.
       The colimit of the $\operatorname{Gr}_{n,m,k, \mathrm{rig}}$ is an \emph{ind-rigid analytic variety}.
       It is denoted by $\operatorname{Gr}_{n}$.
  \end{enumerate}
 \end{Def}

 Theorem~\ref{ClassTopThm} is proven as follows.
 There is a universal vector bundle $\mathcal{E}$ over the Grassmannian such that every vector bundle
 of rank $n$ over a topological space $X$ is the pullback of the universal bundle along some map $X \to \mathrm{Gr}_{n, \mathbb{R}}$.
 If $X$ is paracompact and Hausdorff,
 then any two homotopic maps $X \overset{f}{\underset{g}{\rightrightarrows}} \mathrm{Gr}_n$ give rise to isomorphic bundles
 $f^\ast \mathcal{E}, g^\ast \mathcal{E}$ because topologically, every vector bundle over $X \times [0,1]$ is extended from $X$.
 
 Similar results hold for holomorphic vector bundles over complex Stein spaces and for algebraic vector bundles 
 over smooth affine varieties.
 
 \begin{Thm}[Grauert {\cite{Grauert57b}}]\label{GrauertOka}
  Let $X$ be a complex Stein space. Then there is a bijection
  \begin{equation*}
   [X, \mathrm{Gr}_{n, \mathbb{C}}] \overset{\sim}{\longrightarrow} \{ \text{holomorphic complex vector bundles of rank $n$ over } X \} / \cong
  \end{equation*}
  between homotopy classes of holomorphic maps $X \rightarrow \mathrm{Gr}_{n, \mathbb{C}}$ and isomorphism classes
  of holomorphic vector bundles on $X$.
 \end{Thm}
 The main ingredient in the proof is Grauert's \emph{Oka principle}\index{Oka principle}.
 \begin{Thm}[Grauert's Oka principle, {\cite[Theorem~5.3.1]{Forstneric11}}]
 Let $X$ be a complex Stein space.
  \begin{enumerate}
   \item Every continuous complex vector bundle over $X$ is continuously isomorphic to a holomorphic vector bundle.
   \item Two holomorphic vector bundles over $X$ that are continuously isomorphic are also holomorphically isomorphic. 
  \end{enumerate}
 \end{Thm}
 Grauert and Remmert describe the philosophy behind the Oka principle as follows: \emph{``On a reduced
 Stein space $X$, problems which can be cohomologically formulated have only topological obstructions''} \cite[p.~145]{GrauertRemmert04}.
 
 For algebraic varieties we have the following theorem by Morel. 
 Here, an algebraic vector bundle of rank $n$ over an algebraic variety $X$
 is defined as a locally free $\mathcal{O}_X$-module of rank $n$.
 \begin{Thm}[Morel]\label{MorelThm}
  Let $X$ be a smooth affine variety over a field $k$. Then there is a natural bijection
  \begin{equation*}
   [X, \mathrm{Gr}_{n,k}]_{\mathbb{A}^ 1} \overset{\sim}{\longrightarrow} \{ \text{algebraic $k$-vector bundles of rank $n$ over } X \} / \cong
  \end{equation*}
  between $\mathbb{A}^1$-homotopy classes from $X$ to the infinite Grassmannian $\mathrm{Gr}_n$ (as an ind-variety)
  and isomorphism classes of algebraic vector bundles on $X$. 
 \end{Thm}
 \begin{Rem}
  Fabien Morel proved the theorem in the case that the field $k$ is infinite and perfect and $n \neq 2$ \cite{Morel12}. It was
  proven in general several years later by Asok--Hoyois--Wendt~\cite{AsokHoyoisWendt15}.
 \end{Rem}
 Homotopy invariance is given by Lindel's solution to the geometric case of the Bass--Quillen conjecture.
 The other important ingredient in the proof is an algebraic variant of Gromov's h-principle,
 called the \emph{Brown--Gersten property}\index{Brown--Gersten property} by Morel--Voevodsky
 and \emph{excision} by Asok--Hoyois--Wendt. 
 
\subsection{Representability in $\mathbb{A}^1$-homotopy theory}\label{A1hprin} 
 Asok--Hoyois--Wendt \cite{AsokHoyoisWendt15} show that only homotopy descent with respect to
 the Nisnevich topology and homotopy invariance
 are needed in order for Nisnevich excision to be satisfied (following a trick by Schlichting \cite{Schlichting15}).
 Furthermore, they show that only excision and homotopy invariance are needed to get a classifying space for vector bundles.
 This way they obtain an elegant proof of Morel's Theorem~\ref{MorelThm} which makes the technical assumptions on the base field and
 the rank of the vector bundles superfluous. The base, which had to be an infinite perfect field in Morel's proof,
 may now be any field or any ring which is ind-smooth over a Dedekind ring with perfect residue fields \cite[Theorem~5.2.3]{AsokHoyoisWendt15}. 
 Their approach even generalises to the case of principal bundles under other algebraic groups \cite{AsokHoyoisWendt15b}.
 
 The homotopy theoretic part of the proof by Asok--Hoyois--Wendt carries over to the rigid analytic situation.
 The interesting question becomes then:
 For which subcategory $\mathcal{R}$ of the category of rigid varieties, for which cd structure $\tau$ and which interval
 object $I$ can we show homotopy invariance and excision?
 Once homotopy invariance and excision are ensured, one immediately gets a classifying space in the appropriate homotopy theory 
 $\mathcal{H}_{I, \tau} ( \mathcal{R})$.
 It then remains to determine how this homotopy theory relates to the corresponding homotopy theory $\mathcal{H}_{I, \tau} ( \rigVar )$ 
 on the full category of (smooth) rigid analytic varieties. 
  
 Before coming to the theorems, we need more definitions.
 \begin{Def}[homotopy descent, {\cite[Definition 3.1.1]{AsokHoyoisWendt15}}]\index{descent!homotopy}\index{homotopy descent}
  Let $(\mathcal{C}, \mathcal{T})$ be a small site. A simplicial presheaf $\mathcal{F}$ on $\mathcal{C}$ 
  \emph{satisfies homotopy descent with respect to $\mathcal{T}$} if for every $X \in \operatorname{Ob} ( \mathcal{C} )$ 
  and every covering sieve   $\mathcal{U} \in \operatorname{Cov} (X)$ the map 
  \begin{equation*}
   \mathcal{F} (X) \longrightarrow \operatorname{holim}_{U \in \mathcal{U}} \mathcal{F} (U)
  \end{equation*}
  induced by the restrictions $\mathcal{F} (X) \to \mathcal{F} (U)$ is a weak equivalence.
 \end{Def}
 
 \begin{Remm}\label{Remdescent}
  \begin{enumerate}
   \item Homotopy sheaves satisfy homotopy descent by definition.
   \item\label{Remdesb} Corti\~{n}as--Haesemeyer--Schlichting--Weibel~\cite{CHSW08} show 
       that for a complete, bounded, regular cd structure,
       homotopy descent is equivalent to quasifibrancy. 
       They call a simplicial presheaf $\mathcal{F}$ \emph{quasifibrant} if the local injective
       fibrant replacement $\mathbb{H} ( \cdot , \mathcal{F} )$ gives rise to a weak equivalence on all 
       spaces of sections: $\mathcal{F} (U) \simeq \mathbb{H} ( U , \mathcal{F} )$ for all
       objects $U$. This is useful when one wants to compute homotopy classes. Usually, either cofibrant replacement
       or fibrant replacement is difficult to compute. A simplicial presheaf which satisfies homotopy descent
       is almost as good as fibrant because it still gives rise to the right computations.
  \end{enumerate}
 \end{Remm}
 
 If the Grothendieck topology arises from a cd structure, one can consider excision instead:
 
 \begin{Def}[excision, {\cite[Definition 3.2.1]{AsokHoyoisWendt15}}]\index{excision}\label{Defexcision}
  Let $\mathcal{C}$ be a small category with an initial object $\emptyset$ and $\tau$ a cd structure (see Definition \ref{defcd})
  on $\mathcal{C}$. A simplicial presheaf $\mathcal{F}$ on $\mathcal{C}$ \emph{satisfies $\tau$-excision} if
  \begin{enumerate}
   \item $\mathcal{F} ( \emptyset ) \simeq \ast$ in $\sSet$ and
   \item\label{excisiondist} for each distinguished square
      \begin{equation*}
       \begin{xy}
        \xymatrix{
         B \ar[r] \ar[d] & Y \ar[d] \\
         A \ar[r] & X
        }
       \end{xy}
      \end{equation*}
      in $\tau$, the square
      \begin{equation*}
       \begin{xy}
        \xymatrix{
         \mathcal{F} ( X) \ar[r] \ar[d] & \mathcal{F} (A) \ar[d] \\
         \mathcal{F} (Y) \ar[r] & \mathcal{F} ( B)
        }
       \end{xy}
      \end{equation*}
      is a homotopy pullback square in $\sSet$.
  \end{enumerate}
 \end{Def}
 This property runs under many names: Brown--Gersten property with respect to $\tau$ \cite{MV, Morel12}, 
 $\tau$-flasque \cite{Voevodsky10}, Mayer--Vietoris property with respect to $\tau$ \cite{CHSW08} 
 or $\tau$-excision \cite{AsokHoyoisWendt15}.
 We stick to the term excision, as suggested by Asok--Hoyois--Wendt.
 Excision is easier to check than homotopy descent and if the cd structure is nice enough, 
 the two notions are equivalent by Theorem~\ref{ThmcdVoe}
 below. In nice cases, excision is a weak version of fibrancy as we explained in Remark \ref{Remdescent}\ref{Remdesb}).
 
 \begin{Def}[strictly initial object {\cite[Definition 3.4.6]{Borceux94}}]
  Let $\mathcal{C}$ be a category with an initial object $\emptyset$. The object $\emptyset$ is called \emph{strictly initial}
  if every morphism $X \to \emptyset$ is an isomorphism.
 \end{Def}

 \begin{Thm}[Voevodsky, Asok--Hoyois--Wendt {\cite[Theorem~3.2.5]{AsokHoyoisWendt15}}]\label{ThmcdVoe}
  Let $\mathcal{C}$ be a small category with a strictly initial object and $\tau$ a cd structure on $\mathcal{C}$.
  Assume that
  \begin{enumerate}
   \item every square 
        \begin{equation}\label{distsquare}   
         \begin{xy}\xymatrix{
          B \ar[r]\ar[d] & Y \ar[d]^p\\
          A \ar[r]_e &X
         }\end{xy}
        \end{equation}
        in $\tau$ is cartesian,
   \item pullbacks of squares in $\tau$ exist and belong to $\tau$,
   \item for every square \eqref{distsquare} in $\tau$, the morphism $e$ is a monomorphism,
   \item for every square \eqref{distsquare}        
        in $\tau$, also the square
        \begin{equation*}   
         \begin{xy}\xymatrix{
          B \ar[r]\ar[d]_\Delta & Y \ar[d]^\Delta\\
          B \times _A B \ar[r] &Y \times _X Y
         }\end{xy}
        \end{equation*}
        is in $\tau$.
  \end{enumerate}
  Then $\mathcal{F} \in \mathit{sPSh} ( \mathcal{C})$ satisfies excision with respect to $\tau$ if and only if it 
  satisfies homotopy descent with respect to the Grothendieck topology defined by $\tau$.
 \end{Thm}
 
 \begin{Def}[simplicial resolution, {\cite[p.~88]{MV}}]\index{simplicial resolution}\index{$\operatorname{Sing}^I \mathcal{F}$}\label{singstar}
  Let $\mathcal{C}$ be a small category and $I$ a representable interval object (Definition \ref{definterval}) on $\mathcal{C}$. 
  Denote by $\Delta$ the simplex category \cite[I.1, p.~3]{GoerssJardine}.
  \begin{enumerate}
   \item The cosimplicial presheaf $I^\bullet$ is defined on objects $[n] \in \operatorname{Ob} ( \Delta )$ by
      \begin{align*}
       I^\bullet \colon \Delta &\longrightarrow  \mathit{PSh} ( \mathcal{C}), \\
                      [n] &\longmapsto I^ n.\\
      \end{align*}
      Let $f \in \operatorname{Hom}_\Delta ( [n] , [m] )$. 
      Denote by $\operatorname{pr}_j \colon I^ n \to I$ the projection to the $j$-th coordinate.
      Then $I^\bullet (f)$ is defined by
      \begin{equation*}
       \operatorname{pr}_j \circ I^\bullet (f) = \left\{ 
       \begin{split}
        &\operatorname{pr}_N && \text{if} \enspace 1 \leq N := \min \{ l \in \{ 0 , \dotsc , n \} \mid f(l) \geq j \} \leq n \\
        &I^ n \to \ast \overset{\iota _0 }{\rightarrow} I &&\text{if} \enspace \{ l \in \{ 0 , \dotsc , n \} \mid f(l) \geq j \} = \emptyset \\
        &I^ n \to \ast \overset{\iota _1 }{\rightarrow} I && \text{if} \enspace f(0) \geq j.
       \end{split} \right.
      \end{equation*}
   \item The \emph{simplicial resolution} $\operatorname{Sing}^I \mathcal{F}$ of a simplicial presheaf $\mathcal{F}$
      is the simplicial presheaf defined by
      \begin{equation*}
       \big( \operatorname{Sing}^I \mathcal{F} \big) _n = \big( \mathcal{F} ( X \times I^n ) \big) _n.
      \end{equation*}
      That is, it is the diagonal of the bisimplicial presheaf $\mathcal{F} ( \, \cdot \, \times I^\bullet )$.
  \end{enumerate}
 \end{Def}
 
 \begin{Lemma}[Properties of $\operatorname{Sing}^I$ {\cite[p.~15]{AsokHoyoisWendt15}}]\label{RemarkSing}
  \begin{enumerate}
   \item The construction of the simplicial resolution $\operatorname{Sing} ^I$ gives rise to a natural transformation 
       $\operatorname{id} \to \operatorname{Sing} ^I$ on simplicial presheaves.
       Thus, for every $\mathcal{F} \in \sPSh ( \mathcal{C} )$ there is a map
       \begin{equation*}
        \mathcal{F} \longrightarrow \operatorname{Sing}^I \mathcal{F} .
       \end{equation*}
   \item\label{RemSingb} The simplicial resolution of a simplical presheaf is always $I$-invariant.
   \item\label{RemSingc} For any $I$-invariant presheaf $\mathcal{G}$, the map $\mathcal{F} \to \operatorname{Sing}^I \mathcal{F}$
       induces a weak equivalence
       \begin{equation}
        \mathit{Map} ( \operatorname{Sing}^I \mathcal{F}, \mathcal{G} ) \simeq \mathit{Map} ( \mathcal{F}, \mathcal{G} ).
       \end{equation}
   \item\label{RemSingd} Let $\tau$ be a cd structure. If $\mathcal{F}$ satisfies $\tau$-excision and if $\pi _0 \mathcal{F}$ is $I$-invariant,
       then $\operatorname{Sing}^I \mathcal{F}$ also satisfies excision.
  \end{enumerate}
 \end{Lemma}
 \begin{proof}
  \begin{itemize}
   \item[\ref{RemSingb})] Morel--Voevodsky prove it for simplicial sheaves on a site with interval 
       under the general assumption that the site has enough points \cite[§2 Corollary~3.5]{MV}. 
       Neither the sheaf property nor the assumption that the site
       has enough points go into that particular proof, so we refrain from repeating it here.
   \item[\ref{RemSingc})] This is proven by Morel--Voevodsky, again without using the sheaf property or the
       assumption that the site has enough points \cite[§2 Corollary~3.8]{MV}.
   \item[\ref{RemSingd})] This is \cite[Theorem~4.2.3]{AsokHoyoisWendt15}. \qedhere
  \end{itemize} 
 \end{proof}
 
 Gromov's h-principle can be expressed as a homotopy sheaf property. The reader may therefore think 
 of excision as an $\mathbb{A}^1$-homotopy theoretic version of the h-principle. 
 If, on the other hand, the reader is more familiar with simplicial homotopy theory,
 it might be helpful to know that the Oka prinicple can be expressed as a fibrancy condition in a
 suitable model category of complex analytic spaces. This was shown by
 L\'{a}russon~\cite{Larusson03, Larusson04, Larusson05}.
 He calls manifolds (or, more generally, complex spaces) that satisfy an Oka principle \emph{Oka manifolds} (Oka spaces).
 He constructs a model category of simplicial presheaves containing the category of complex manifolds, such that
 Stein manifolds are cofibrant and Oka manifolds are fibrant. It is an interesting question to find other
 Oka manifolds (or Oka spaces), one is looking for representable simplicial presheaves that satisfy
 the Oka principle.
 The branch of complex analysis investigating this question is called Oka theory.
 We recommend Forstneri\v{c}'s book~\cite{Forstneric11} for further information about Oka theory.
 
\subsection{Classifying spaces}
 We show that the infinite Grassmannian $\operatorname{Gr}_n$ is 
 $\mathbb{A}^1_\mathrm{rig}$-homotopy equivalent to the simplicial classifying space of $\operatorname{GL}_n$. 
 This will be needed for our main theorem, Theorem~\ref{ThmClass}.
 
 \begin{Def}[Classifying space of a simplicial sheaf of groups {\cite[§4.1]{MV}}]\label{DefBGs}
  Let $\mathcal{G}$ be a sheaf of simplicial groups (or, more generally: monoids) on a site $\mathcal{C}$. 
  Hence, to each $X \in \operatorname{Ob} ( \mathcal{C})$ and each $n \in \mathbb{N}$, the sheaf
  $\mathcal{G}$ assigns a group (monoid) $G_{X,n}$. Viewing this group (monoid) as a category, it has a
  nerve $N (G_{X,n})$.
  Now the \emph{classifying space}\index{classifying space!simplicial} $B \mathcal{G}$ of $\mathcal{G}$
  is the diagonal simplicial sheaf of the bisimplicial sheaf
  \begin{equation*}
   X \longmapsto \big( n \mapsto N( G_{X,n}) \big) .
  \end{equation*}
  This means, for $X \in \operatorname{Ob} ( \mathcal{C})$,
  \begin{equation*}
   B \mathcal{G} (X) = \big( (G_{X,n} )^ n \big) _{n \in \mathbb{N}}
  \end{equation*}
  where $(G_{X,0} )^ 0 := \ast$.
 \end{Def}
 For the definitions of the nerve and of the classifying space of a topological category, we refer to Segal {\cite[§2]{Segal68}}.
 
 \begin{Def}[Universal torsor, cf.~{\cite[§4, Example~1.11]{MV}, \cite[p.~12]{Morel06}}]
  For a simplicial sheaf of groups $\mathcal{G}$, define a simplicial sheaf of groups $E \mathcal{G}$ as the diagonal
  of the bisimplicial sheaf
  \begin{equation*}
   X \longmapsto \big( (n,m) \mapsto E (G_{X,n})_m \big) .
  \end{equation*}
  The sheaf $\mathcal{G}$ is now a subsheaf of groups of $E \mathcal{G}$. For $X \in \operatorname{Ob} ( \mathcal{C} ) , n \in \mathbb{N}$
  we again have an action of $(G_{X,n} )_n$ on $E (G_{X,n} )_n$ given by 
  \begin{equation*}
   (g_0 , \dotsc , g_n ).g := (g_0, g_1 , \dotsc , g_{n-1}, g_ng ).
  \end{equation*}
  This defines an isomorphism from the quotient of sheaves $E \mathcal{G} / \mathcal{G}$ to $B \mathcal{G}$.
  The morphism 
  \begin{equation}
   E \mathcal{G} \longrightarrow B \mathcal{G}
  \end{equation} 
  is called the \emph{universal $\mathcal{G}$-bundle}\index{universal $\mathcal{G}$-bundle} 
  or \emph{universal $\mathcal{G}$-torsor}\index{universal torsor}.
 \end{Def}
 For the definition of the classifying space of a topological group, we refer to May {\cite[chapter~23, §8]{May}}.
 An explicit construction, similar to the one we use, can also be found in May's book 
 {\cite[Chapter~16, §5 and chapter~23, §8]{May} or Segal's article \cite[§3]{Segal68}}.
 
 The simplicial classifying space $BG$ in fact classifies principal $G$-bundles:
 \begin{Deff}[{\cite[pp.~127f]{MV}}]\index{classifying space!simplicial}
  \begin{enumerate}
   \item Let $\mathcal{C}$ be a site and $\mathcal{G}$ a sheaf of simplicial groups on $\mathcal{C}$.
       A right action of $\mathcal{G}$ on a simplicial sheaf $\mathcal{F}$ is called 
       \emph{(categorically) free}\index{group action!(categorically) free}
       if the morphism
       \begin{align*}
        \mathcal{F} \times \mathcal{G} &\longrightarrow \mathcal{F} \times \mathcal{F} \\
        (g,x) &\longmapsto (x.g,x)
       \end{align*}
       is a monomorphism.
   \item If $\mathcal{G}$ acts on $\mathcal{F}$ from the right, then the quotient
       $\mathcal{F} / \mathcal{G}$ is defined as the coequaliser of
       \begin{equation*}
        \mathcal{F} \times \mathcal{G} \overset{\operatorname{pr}_1}{\underset{\text{action}}{\rightrightarrows}} \mathcal{F} .
       \end{equation*}
   \item A \emph{principal $\mathcal{G}$-bundle} or \emph{$\mathcal{G}$-torsor} over $\mathcal{F}$ 
       is a morphism $p\colon \mathcal{E} \to \mathcal{F}$ together with a free
       right action of $\mathcal{G}$ on $\mathcal{E}$ over $\mathcal{F}$, such that $\mathcal{E} / \mathcal{G} \to \mathcal{F}$
       is an isomorphism.
  \end{enumerate}
 \end{Deff}
 \begin{Lemma}[{\cite[§4, Lemma~1.12]{MV}}]\label{sGLnclassifies}
  Let $\mathcal{C}$ be a site and $\mathcal{G}$ a sheaf of simplicial groups on $\mathcal{C}$.
  Let $p\colon \mathcal{E} \to \mathcal{F}$ be a principal $\mathcal{G}$-bundle. Then there are morphisms
  \begin{itemize}
   \item $\omega \colon \mathcal{F}' \to \mathcal{F}$ such that $\mathcal{F}' (U) \to \mathcal{F}(U)$ 
       is both a fibration and a weak equivalence in $\sSet$
       for all $\mathcal{U} \in \operatorname{Ob} \mathcal{C}$ and
   \item $\varphi \colon \mathcal{F}' \to B \mathcal{G}$
  \end{itemize}
  such that the pullback of the universal $\mathcal{G}$-torsor along $\varphi$ and the pullback of $p$ along $\omega$
  coincide:
  \begin{equation*}
   \begin{xy}
    \xymatrix{
     E \mathcal{G} \ar[d] & \varphi ^ \ast E \mathcal{G} = \omega ^ \ast \mathcal{E} \ar[d] \ar[r] \ar[l] & \mathcal{E} \ar[d]^p \\
     B \mathcal{G} & \mathcal{F}' \ar[l]^\varphi \ar[r]_\omega & \mathcal{F}
    }
   \end{xy}
  \end{equation*} 
 \end{Lemma}
 This means that in a model structure on $\mathit{sSh} ( \mathcal{C})$ where
 \begin{itemize}
  \item morphisms $\mathcal{F}' \to \mathcal{F}$ such that $\mathcal{F} ' (U) \to \mathcal{F} (U)$ is a weak equivalence
     for all $U \in \operatorname{Ob} \mathcal{C}$ are themselves weak equivalences and
  \item morphisms $\mathcal{F}' \to \mathcal{F}$ such that $\mathcal{F} ' (U) \to \mathcal{F} (U)$ is a fibration
     for all $U \in \operatorname{Ob} \mathcal{C}$ are themselves fibrations,
 \end{itemize}
 the following hold:
 \begin{enumerate}
  \item The classifying space $B \mathcal{G}$ is in fact the homotopy quotient $\ast /^h G$. By this, we mean that
     the total space $E \mathcal{G}$ of the universal $\mathcal{G}$-bundle is contractible and $E \mathcal{G} \to B \mathcal{G}$
     is a fibration.
  \item The classifying space $B \mathcal{G}$ really classifies principal $\mathcal{G}$-bundles in $\mathit{sSh} ( \mathcal{C})$.
 \end{enumerate}
 
 We want to identify the infinite Grassmannian as the classifying space of $\operatorname{GL}_n$.
 To be precise, we want to prove the rigid analytic version of the following $\mathbb{A}^1$-homotopy theoretic result.
 \begin{Prop}[{\cite[p.~138, Proposition~3.7]{MV}}]
  In the $\mathbb{A}^1$-homotopy category, we have $B\operatorname{GL}_n \cong \operatorname{Gr}_n$.
 \end{Prop}
 A proof was written down by Peter Arndt. It generalises the proof for $n = 1$ by Naumann--Spitzweck--{\O}stv{\ae}r
 \cite[Corollary~2.5]{NaumannSpitzweckOstvaer09} and works also in the setting of rigid analytic varieties.
 One writes $\operatorname{Gr}_n$ as a quotient of the Stiefel variety by a free $\mathrm{GL}_n$-action
 and shows that the Stiefel variety is contractible. This approach is originally due to Morel, 
 cf.~\cite[Example~2.1.8, Lemma~4.2.5]{Morel06}.
 \begin{Prop}[Arndt]\label{GrBGL}\index{Grassmannian}
  We have $B\operatorname{GL}_n \cong \operatorname{Gr}_{n}$ in the $\mathbb{A}^1_\mathrm{rig}$-homotopy category.
  Consequently, $B\operatorname{GL}_n \cong \operatorname{Gr}_{n}$ in the $\mathbb{B}^1$-homotopy category.
 \end{Prop}
 \begin{proof} 
  Let $\mathcal{R}$ be a subcategory of the category of smooth rigid analytic varieties containing
  the affinoids of good reduction. 
  We write $\operatorname{Gr}_n$ as a quotient of the \emph{Stiefel space}\index{Stiefel space} 
  $\operatorname{Fr}_n$ by a free $\mathrm{GL}_n$-action
  and show that $\operatorname{Fr}_n$ is contractible in $\mathcal{M} _{\mathbb{B} ^1} ( \mathcal{R} )$ 
  and $\mathcal{M} _{\mathbb{A} ^1_\mathrm{rig}} ( \mathcal{R} )$. 
    
  By \cite[Proposition~I.2.13]{Lam06}, we know that for each smooth $k$-algebra $A$ we have
  $\operatorname{GL}_n (A) \cong \operatorname{GL}_n (k) \otimes _k A$,
  hence the sheaf $\operatorname{GL}_n$ is representable
  by $\operatorname{GL}_n (k)$.
  
  Let $\operatorname{Gr}_{n,m}$ be a finite Grassmannian.
  We define the space $\operatorname{Fr}_{n,m}$ of $n$-frames in $\mathbb{A}^ m_\mathrm{rig}$, building it up from pieces $M_J$.
  We think of the space of $n$-frames as the space of $n \times m$-matrices of rank $n$.
  
  During this proof we denote
  \begin{equation*}
   \mathcal{O} ( \mathbb{A}^l _\mathrm{rig} ) =: k \langle \langle T_1 , \dotsc , T_l \rangle \rangle .
  \end{equation*}
  
  For each subset $J = \{ j_1, \dotsc , j_n \} \subset \{ 1 , \dotsc , m \}$ of order $n$ we define a rigid analytic space $M_J$ as
  \begin{equation*}
   M_J = \bigg( \operatorname{Spec} \big( k [T_{11} , \dotsc , T_{1m} , T_{21}, \dotsc , T_{2m} , \dotsc , T_{nm} , T] / 
   ( (\operatorname{det} ((T_{ij})_J)T)-1) \big) \bigg)^\mathrm{an} 
  \end{equation*}
  where $(T_{ij})_J$ is defined as the $n \times n$ minor of 
  $(T_{ij})_{\substack{i= 1, \dotsc , n \\  j= 1 , \dotsc , m}}$ given by $J$, i.\,e.,
  \begin{equation*}
   (T_{ij})_J := (T_{ij}) _{\substack{i= 1, \dotsc , n \\  j = j_1 , \dotsc , j_n}} .
  \end{equation*}
  The spaces $M_J$, the index set $J$ running through all  
  subsets of $\{ 1 , \dotsc , m \}$ with $n$ elements,
  are supposed to cover the space of $n$-frames that we want to define.
  We think of $M_J$ as the space of $n \times m$-matrices such that the $n \times n$-minors specified by $J$
  are of rank $n$. 
  For a rigorous construction, we need to define the ``intersections'' of $M_{J_1} , \dotsc , M_{J_r}$.
  For $r = 2$, we define
  \begin{equation*}
   M_{J, J'} = \bigg( \operatorname{Spec} \bigg( k [T_{11} , \dotsc , T_{nm} ,T,T'] / 
   \big(  (\operatorname{det} ((T_{ij})_J)T)-1 , (\operatorname{det} ((T_{ij})_J)T')-1  \big) \bigg) \bigg) ^\mathrm{an}
  \end{equation*}
  which we think of as the $n \times m$-matrices such that both the $n \times n$-minors specified by $J$
  and the $n \times n$-minors specified by $J'$ are of rank $n$. 
  Higher ``intersections'' $M_{J_1 , \dotsc , J_r}$ are defined analogously.
  We have open immersions $M_{J,J'} \to M_J$, given on the level of rings by
  \begin{align*}
   T_{ij} &\mapsto T_{ij} \\ 
   T &\mapsto T,
  \end{align*}
  and $M_{J,J'} \to M_{J'}$, given on the level of rings by
  \begin{align*}
   T_{ij} &\mapsto T_{ij} \\ 
   T &\mapsto T'
  \end{align*}
  and analogously $M_{J,J',J''} \to M_{J,J'}$, $M_{J,J',J''} \to M_{J,J''}$, $M_{J,J',J''} \to M_{J',J''}$
  and so on.
  Altogether, they fit into a diagram $D$:
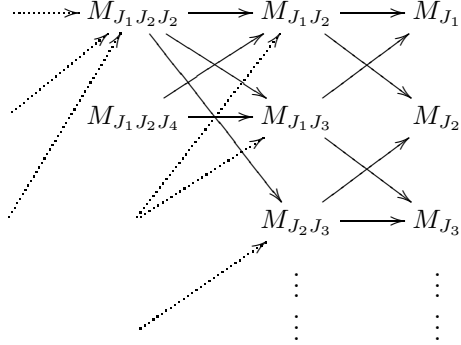
  \begin{equation*}
   \begin{xy}
    \xymatrix{
    & M_{J_1J_2J_2} \ar[r] \ar[dr] \ar[ddr] \ar@{<.}[l]\ar@{<.}[ld]\ar@{<.}[ldd] & M_{J_1J_2} \ar[r] \ar[dr] \ar@{<.}[ldd] & M_{J_1} \\
    & M_{J_1J_2J_4} \ar[ur] \ar[r]                                               & M_{J_1J_3} \ar[ur] \ar[dr] \ar@{<.}[ld] & M_{J_2} \\
    &                                               & M_{J_2J_3} \ar[r] \ar[ur] \ar@{<.}[dl] \ar@{}[d]|{\vdots} & M_{J_3} \ar@{}[d]|{\vdots} \\
    &                                                            & \vdots                                  & \vdots 
    }
   \end{xy}
  \end{equation*}
  Now we define $\operatorname{Fr}_{n,m} := \operatorname{colim} D$ as the ``union'' of all the spaces $M_J$.
  
  The group $\operatorname{GL}_n$ operates on $\operatorname{Fr}_{n,m}$ as follows.
  The action of $\operatorname{GL}_n$ on the space of $n \times m$-matrices by multiplication from the left
  defines an action of $\operatorname{GL}_n$ on $\mathbb{A}^{nm}_\mathrm{rig}$. 
  It leaves the ideals $( (\operatorname{det} ((T_{ij})_J)T)-1)$
  invariant, hence it restricts to each of the $M_J$. As the morphisms $M_{J_1 , \dotsc , J_r} \to M_{J_1 , \dotsc , J_{r-1}}$
  are $\operatorname{GL}_n$-equivariant, we get an action of $\operatorname{GL}_n$ on $\operatorname{Fr}_{n,m}$.
  This action is free by the ``full rank'' condition.
  
  There are $\operatorname{GL}_n$-equivariant embeddings $\operatorname{Fr}_{n,m} \to \operatorname{Fr}_{n,m+n}$
  given by
  \begin{align*}
   \iota _m\colon \mathbb{A}^{nm}_\mathrm{rig} &\longmapsto \mathbb{A}^{nm}_\mathrm{rig} \times \mathbb{A}^{nn}_\mathrm{rig} \\
   x &\longmapsto (x,0).
  \end{align*}
  As $\iota _m$ ``does not change the rank of a matrix'', which means it leaves the ideals 
  \mbox{$( (\operatorname{det} ((T_{ij})_J)T)-1)$}
  invariant, it restricts to $\operatorname{Fr}_{n,m} \to \operatorname{Fr}_{n,m+n}$. Equivariance under the action
  by $\operatorname{GL}_n$ is immediate.
  Fix $m \bmod n$ and define
  \begin{equation*}
   \operatorname{Fr} _n := 
   \operatorname{colim} ( \operatorname{Fr}_{n,m} \to \operatorname{Fr}_{n,m+n} \to \operatorname{Fr}_{n,m+2n} \to \dotsb ).
  \end{equation*}
  The infinite space of $n$-frames $\operatorname{Fr}_n$ depends on the choice of $m \bmod n$, but this does not affect the statement.
  As the embeddings $\operatorname{Fr}_{n,m} \to \operatorname{Fr}_{n,m+n}$ are $\operatorname{GL}_n$-equivariant,
  the colimit $\operatorname{Fr}_n$ inherits the free $\operatorname{GL}_n$-action.
  
  The only statement left to show is that $\operatorname{Fr}_n$ is contractible.
  We show that each embedding $\iota _m \colon \operatorname{Fr}_{n,m} \to \operatorname{Fr}_{n,m+n}$ is 
  $\mathbb{A}^1_\mathrm{rig}$-homotopic to a constant map. 
  Then, as the colimit over a chain of constant maps is contractible, we are done.
  
  An explicit $\mathbb{A}^1_\mathrm{rig}$-homotopy from $\iota _m$ to a constant map can be written down as follows:
  \begin{equation*}
   H \colon \mathbb{A}^1_\mathrm{rig} \times \operatorname{Fr}_{n,m} \longrightarrow \operatorname{Fr}_{n,m+n} 
  \end{equation*}
  If $k$ is algebraically closed, $H$ is given on matrices by
  \begin{equation*}
   H \colon (t,(x_{ij})_{ij}) \longmapsto \left( t \cdot (x_{ij})_{ij} \left\vert (1-t) \cdot \left( 
   \begin{smallmatrix}1 && 0 \\ & \ddots & \\ 0 && 1 \end{smallmatrix} \right. \right) \right).
  \end{equation*}
  In general, it is given on the $M_J$ (and similarly on the $M_{J,J'}$ and so forth) by
  \begin{align*}
   k \langle\langle T_{11} , \dotsc , T_{n(m+n)} , T \rangle \rangle / 
   ( \operatorname{det}_J ) &\longrightarrow k \langle \langle T_{11} , \dotsc , T_{nm} , T' \rangle \rangle \\
   T_{ij} &\longmapsto T' T_{ij} \qquad \ \text{if } j = 1, \dotsc , m, \\
   T_{i  m+j} & \longmapsto  0 \qquad \qquad \text{if } i \neq j,\\
   T_{i  m+i} & \longmapsto(1-T') T_{i m+i} . 
  \end{align*}
  Here, $\operatorname{det}_J := ((T_{ij})_J)T)-1$.
  We need to check that $H$ really ends up in $\operatorname{Fr}_{n,m+n}$. The reason is that, locally, one of $t$ and $1-t$
  is always a unit, hence one of  $t \cdot (x_{ij})_{ij}$ and 
  $(1-t) \cdot \left(  \begin{smallmatrix}1 && 0 \\ & \ddots & \\ 0 && 1 \end{smallmatrix} \right)$
  is of full rank, hence 
  $\left( t \cdot (x_{ij})_{ij} \left\vert (1-t) \cdot \left( 
  \begin{smallmatrix}
   1 && 0 \\ & \ddots & \\ 0 && 1 
  \end{smallmatrix} \right. \right) \right)$
  is of full rank.
  
  The map $H$ defines an $\mathbb{A}^1_\mathrm{rig}$-homotopy between $\operatorname{id}_{\operatorname{Fr}_n}$ 
  and $\operatorname{Fr}_n \to \ast$.
  Consequently, the Stiefel variety $\operatorname{Fr}_n$ is contractible in 
  $\mathcal{M} _{\mathbb{A} ^1_\mathrm{rig}} ( \rigVar )$ and hence
  \begin{equation*}
   \operatorname{Fr}_n / \operatorname{GL}_n \simeq B \operatorname{GL}_n
  \end{equation*}
  in $\mathcal{M} _{\mathbb{A} ^1_\mathrm{rig}} ( \rigVar )$.
  It is clear that $\operatorname{Fr}_n / \operatorname{GL}_n = \operatorname{Gr}_n$, hence 
  \begin{equation*}
   \operatorname{Gr}_n \simeq B \operatorname{GL}_n
  \end{equation*}
  in $\mathcal{M} _{\mathbb{A} ^1_\mathrm{rig}} ( \rigVar )$.
  Therefore the same holds for the sheaves represented by $\operatorname{Gr}_n$, respectively by $B \operatorname{GL}_n$
  and still so after restriction to $\mathcal{R}$:
  \begin{equation*}
   \operatorname{Gr}_n (X) \simeq B \operatorname{GL}_n (X)
  \end{equation*}
  for all $X \in \mathcal{R}$. Hence
  \begin{equation*}
   \operatorname{Gr}_n \simeq B \operatorname{GL}_n
  \end{equation*}
  in $\mathcal{M} _{\mathbb{A} ^1_\mathrm{rig}} ( \mathcal{R})$.
  As weak equivalences in $\mathcal{M} _{\mathbb{A} ^1_\mathrm{rig}} ( \mathcal{R})$ remain weak equivalences after further
  localisation of the model structure, we also get 
  $\operatorname{Gr}_n \simeq B \operatorname{GL}_n$ in $\mathcal{M} _{\mathbb{B} ^1} ( \mathcal{R})$.
 \end{proof}
 
 \begin{Rem}
  The ``spaces'' $\operatorname{Fr}_{n,m}$, $\operatorname{Fr}_n$, $\operatorname{Gr}_{n,m}$ and
  $\operatorname{Gr}_n$ are in fact presheaves that may not be representable on $\mathcal{R}$.
  The spaces $\operatorname{Fr}_{n,m}$ and $\operatorname{Gr}_{n,m}$ are at least smooth rigid analytic varieties,
  the finite Grassmannians $\operatorname{Gr}_{n,m}$ are even quasicompact. The infinite spaces
  $\operatorname{Fr}_n$ and $\operatorname{Gr}_n$ are ind-rigid analytic varieties.
  They are thus ``spaces'' in some sense, albeit not necessarily objects of $\mathcal{R}$.
 \end{Rem}
 
\subsection{Homotopy invariance implies a classification of vector bundles}\sectionmark{Hoinvariance implies classification}\label{secClass}
 The following proposition is \cite[Theorem~5.1.3]{AsokHoyoisWendt15} transferred into our setting.
 \begin{Prop}\label{PropClass}
  Let $\mathcal{R}$ be a subcategory of the category of rigid analytic varieties. Assume that $\mathcal{R}$ is small
  and has a strictly initial object $\emptyset$. Let $I$
  be a representable interval object on $\mathcal{R}$ and $\tau$ a
  cd structure on $\mathcal{R}$.
  Let $\mathcal{F} \in \sPSh ( \mathcal{R})$ be a simplicial presheaf satisfying $\tau$-excision and assume that the associated
  presheaf of pointed sets $\pi _0 \mathcal{F}$ is $I$-invariant on $\mathcal{R}$. Let $R_\tau$ be a fibrant replacement
  functor in $\mathcal{M}_\tau ( \mathcal{R})$. Then $R_\tau \operatorname{Sing}^I \mathcal{F}$ is fibrant in 
  $\mathcal{M}_{I,\tau} ( \mathcal{R})$ and for all $U \in \operatorname{Ob} ( \mathcal{R})$ the canonical map
  \begin{equation*}
   \pi _0 \mathcal{F} (U) \longrightarrow [U, \mathcal{F} ]_{I, \tau}
  \end{equation*}
  is a bijection of pointed sets.
 \end{Prop}
 \begin{proof} 
  The proof is the same as in \cite[Theorem~5.1.3]{AsokHoyoisWendt15}, only easier:
  Since $\operatorname{Sing}^I \mathcal{F}$ is $I$-invariant by Lemma~\ref{RemarkSing}\ref{RemSingb}) and since
  $R_\tau \operatorname{Sing}^I \mathcal{F} (U) = \operatorname{Sing}^I \mathcal{F}(U)$,
  also $R_\tau \operatorname{Sing}^I \mathcal{F}$ is $I$-invariant. Being fibrant in $\mathcal{M}_\tau ( \mathcal{R})$
  and $I$-invariant, it is $I$-local in $\mathcal{M}_\tau ( \mathcal{R})$,
  hence fibrant in $\mathcal{M}_{I,\tau} ( \mathcal{R})$. 
  Therefore we get for each $U \in \operatorname{Ob} ( \mathcal{R})$ a bijection of pointed sets
  \begin{equation*}
   \pi _0 \operatorname{Sing}^I \mathcal{F} (U) \overset{\star}{\cong}
   \pi _0 R_\tau \operatorname{Sing}^I \mathcal{F} (U).
  \end{equation*}
  As, by assumption,
  \begin{equation*}
   \pi _0 \mathcal{F} (U) \overset{\diamond}{\cong} \pi _0 \operatorname{Sing}^I \mathcal{F} (U),
  \end{equation*}
  we get the following identifications:
  \begin{align*}
   \pi _0 \mathcal{F} (U) &\overset{\diamond}{\cong} \pi _0 \operatorname{Sing}^I \mathcal{F} (U) \overset{\star}{\cong}
   \pi _0 R_\tau \operatorname{Sing}^I \mathcal{F} (U) \cong \pi _0 \operatorname{Hom} (U, R_\tau \operatorname{Sing}^I \mathcal{F})\\
   &= \pi _0 \mathit{Map}_{I, \tau} (U, \mathcal{F}) = [U, \mathcal{F} ]_{I,\tau}.\qedhere
  \end{align*}
 \end{proof}
  
 The next theorem is the rigid analytic version of \cite[Theorem~5.2.3]{AsokHoyoisWendt15}. It basically states that
 if homotopy invariance holds, then there is a homotopy classification of vector bundles.
 \begin{Thm}\label{ThmClass}
  Let $\mathcal{R}$ be a subcategory of the category of rigid analytic varieties. Assume that $\mathcal{R}$ is small
  and has a strictly initial object $\emptyset$. 
  Let $I \in \{ \mathbb{B}^1 , \mathbb{A}^1_\mathrm{rig} \}$ be a representable interval object on $\mathcal{R}$. 
  For some $n \in \mathbb{N}$, assume that $\vect ^n$ satisfies $I$-invariance on $\mathcal{R}$.
  Let $\tau _{cdG}$ be the cd structure which consists of finite admissible coverings with respect to the G-topology. 
  Assume that 
  \begin{itemize}
   \item $\tau = \tau _{cdG}$ or
   \item $\tau$ is a cd structure satisfying the assumptions of Theorem~\ref{ThmcdVoe}, $\tau$ is finer than $\tau _{cdG}$
       and $\vect ^n$ satisfies excision with respect to $\tau$.
  \end{itemize}
  Then the infinite Grassmannian classifies $\vect ^n$ in the $(I, \tau )$-homotopy category of $\mathcal{R}$:
  For each $X \in \operatorname{Ob} ( \mathcal{R} )$ there is a natural bijection
      \begin{equation*}
       \vect ^n(X) \cong [ X, \mathrm{Gr}_n ]_{I, \tau}.
      \end{equation*}
 \end{Thm}
 \begin{proof}
  Let $\Vect ^ n$ be the functor ``big vector bundles'' as defined in \ref{Defstrictvect}.
  By assumption, $\vect ^n$ satisfies excision with respect to $\tau$.
  If $\tau = \tau _{cdG}$, then $\vect ^n$ satisfies descent, hence also excision,
  with respect to $\tau$ by Bosch--G\"{o}rtz~\cite[Theorem~3.1]{BoschGoertz98}.
  
  We claim that the assignment
  $U \mapsto B \Vect ^n (U)$ satisfies $\tau$-excision.
  Let
  \begin{equation*}
   \begin{xy}
   \xymatrix{
    U \times _X V \ar[r] \ar[d] & U \ar[d]^f \\
    V \ar[r]_g & X
   }
   \end{xy}
  \end{equation*}
  be a $\tau$-distinguished square in $\mathcal{R}$.
  Apply Quillen's Theorem~B in the version \cite[Theorem~$B_1$ (3.5)]{BarwickKan11} to the zigzag of groupoids
  \begin{equation*}
   \Vect ^n (U) \overset{f^\ast}{\longrightarrow} \Vect ^n (U \times _X V) \overset{g^\ast}{\longleftarrow} \Vect ^n (V).
  \end{equation*}
  Barwick--Kan's condition $B_1$ is trivially fulfilled for groupoids. Hence 
  \begin{equation}\label{Madrid}
   \begin{xy}
    \xymatrix{
     f^\ast \Vect ^n (U) \downarrow g ^\ast \Vect ^n (V) \ar[d]\ar[r] & \Vect ^n (U) \ar[d] \\
     \Vect ^n (V) \ar[r] & \Vect ^n (U \times _X V)
    }
   \end{xy}
  \end{equation}
  is homotopy cartesian by \cite[Theorem~$B_1$ (3.5)]{BarwickKan11}.
  Here, $f^\ast \Vect ^n (U) \downarrow g ^\ast \Vect ^n (V)$ is the comma category whose objects are
  isomorphisms $f^\ast \Vect ^n (U) \overset{\sim}{\longrightarrow} g^\ast \Vect ^n (V)$ and
  whose morphisms are commutative diagrams.
  The square \eqref{Madrid} is by definition homotopy cartesian if and only if it is so after taking nerves.
  If $\tau = \tau _{cdG}$, then $f$ and $g$ are the inclusions of $U$, respectively of $V$, into $X$ and
  $f^\ast = \operatorname{res}^U_{UV}$ and $g^\ast = \operatorname{res}^V_{UV}$ are the restrictions of vector bundles on $U$,
  respectively $V$, to vector bundles on $U \times _X V = U \cap V$. In this case there are isomorphisms
  \begin{equation*}
   \operatorname{res}^U_{UV} \Vect ^n (U) \downarrow \operatorname{res}^V_{UV} \Vect ^n (V) \overset{\star}{\cong} 
   \Vect ^n (U) \times _{\Vect ^n (U \cap V)} \Vect ^n (V) \overset{\diamond}{\cong} \Vect ^n (X)
  \end{equation*}
  by definition. If $\tau \neq \tau _{cdG}$ then $\star$ remains an isomorphism and, by excision, 
  $\diamond$ is at least a weak equivalence (after taking nerves).
  Putting everything together, the diagram
  \begin{equation*}
   \begin{xy}
    \xymatrix{
     B \Vect ^n (X) \ar[r] \ar[d] & B \Vect ^n (U) \ar[d] \\
     B \Vect ^n (V) \ar[r] & B \Vect ^n (U \times _X V)
    }
   \end{xy}
  \end{equation*}
  is homotopy cartesian, proving the claim.
  
  By assumption, $\vect ^n$ is $I$-invariant on $\mathcal{R}$, hence $\pi _0 B \Vect ^n$ is $I$-invariant on $\mathcal{R}$.
  Applying Proposition~\ref{PropClass} to the simplicial presheaf $B \Vect ^n$ gives
  \begin{equation*}
   \vect ^ n (X) = [ X , B \Vect ^ n ]_{I, \tau}
  \end{equation*}
  for all $X \in \operatorname{Ob} (\mathcal{R})$.  
  
  Let now $\mathcal{P}^n$ be the functor ``big $\operatorname{GL}_n$-principal bundles'' defined in \ref{Defstrictvect}.
  By Lemma~\ref{sGLnclassifies}, the simplicial classifying space $B \operatorname{GL}_n$ really classifies 
  $\operatorname{GL}_n$-principal bundles on $\mathcal{R}$. Hence, for all $X \in \mathcal{R}$, we have
  \begin{equation*}
   B \operatorname{GL}_n (X) \simeq B \mathcal{P}^ n.
  \end{equation*}
  Furthermore, $\Vect ^ n (X)$ and $\mathcal{P}^ n(X)$ are equivalent as topological groupoids for each $X \in \mathcal{R}$.
  In particular, we get that
  \begin{equation*}
   B \mathcal{P}^ n (X) \simeq B \Vect ^ n(X).
  \end{equation*}
  As vector bundles are $\tau _{cdG}$-locally trivial by assumption and thus $\tau$-locally
  trivial, this gives a $\tau$-local weak equivalence  
  \begin{equation*}
   B \operatorname{GL}_n \simeq B \Vect ^ n.
  \end{equation*}

  By Proposition~\ref{GrBGL}, $B \operatorname{GL}_n \simeq \operatorname{Gr}_n$ in $\mathcal{M}_{I, \tau } ( \mathcal{R})$.
  Hence, we get for all $X \in \operatorname{Ob} (\mathcal{R})$ that
  \begin{equation*}
   \vect ^n(X) \cong [ X, \mathrm{Gr}_n ]_{I, \tau}. \qedhere
  \end{equation*}
 \end{proof}
  
 With $\mathbb{A}^1_\mathrm{rig}$ as an interval, we get the following classification theorem:
 \begin{Thm}\label{ThmClassLines}
  Let $\mathcal{R}$ be the category of smooth $k$-rigid analytic quasi-Stein varieties.
  Let $\tau _{cdG}$ be the cd structure generated by
  squares
  \begin{equation*}
   \begin{xy}\xymatrix{
    U \cap V \ar[r]\ar[d] & V \ar[d]^\cap \\
    U \ar[r]_\subset &X
   }\end{xy}
  \end{equation*}
  where $U, V, X \in \operatorname{Ob} ( \mathcal{R} )$ and $\{ U \to X , V \to X \}$ is an admissible covering of $X$.
  Then the assumptions of Theorem~\ref{ThmClass} are satisfied for $n = 1$ and $I = \mathbb{A}^1_\mathrm{rig}$. 
  Consequently there is a natural bijection
  \begin{equation*}
   \operatorname{Pic} (X) \cong [ X, \mathbb{P}^\infty _\mathrm{rig} ]_{ \mathbb{A}^1_\mathrm{rig}, \tau _{cdG}}.
  \end{equation*}
 \end{Thm}
 \begin{proof}
  The object $\mathbb{A}^1_\mathrm{rig}$ is a representable interval object on the category of 
  smooth rigid analytic quasi-Stein varieties.
  Vector bundles satisfy descent with respect to the G-topology by Bosch--G\"{o}rtz~\cite[Theorem~3.1]{BoschGoertz98},
  hence they also satisfy homotopy descent with respect to the G-topology.
  Every vector bundle on a quasi-Stein variety has a finite local trivialisation by Ben's Theorem~\ref{SteinfinG}.
  Homotopy invariance holds by Proposition~\ref{rigA1inv}. Hence, by Theorem~\ref{ThmClass},
  \begin{equation*}
   \operatorname{Pic} (X) \cong [ X, \mathrm{Gr}_1 ]_{I, \tau}.
  \end{equation*}
  As $\mathrm{Gr}_1 = \operatorname{P}^\infty _\mathrm{rig}$, we are finished.
 \end{proof}
 
 On the other hand, choosing $\mathbb{B}^ 1$ as an interval object, homotopy invariance only holds over smooth
 affinoid varieties of sufficiently good reduction. The resulting data do not satisfy the assumptions of Theorem~\ref{ThmClass},
 as the following examples show.
 \begin{Exx}\label{Exxclass}
  \begin{enumerate}
   \item Let $n = 1, I = \mathbb{B}^1$ and $k$ algebraically closed of residue characteristic zero.
      Let $\mathcal{R}$ be the category of smooth affinoid varieties of semistable canonical reduction. 
      The interval object $\mathbb{B}^1$ is representable: If $X \in \operatorname{Ob} ( \mathcal{R})$ is a smooth
      affinoid variety of semistable canonical reduction, then $X \times \mathbb{B}^1$ also has these properties.
      Let $\tau = \mathrm{Zar}$ be the cd structure generated by squares
      \begin{equation*}
       \begin{xy}\xymatrix{
        U \cap V \ar[r]\ar[d] & V \ar[d]^\cap \\
        U \ar[r]_\subset &X
       }\end{xy}
      \end{equation*}
      where $X \in \operatorname{Ob}$ and the maps $U \to X$ and $V \to X$ are inclusions of Zariski open subsets.
      We saw in Corollary~\ref{vbequiv} that on an affinoid variety, every vector bundle with respect to the G-topology
      is isomorphic to a vector bundle with respect to the Zariski topology. As every Zariski vector bundle is in particular
      a G-vector bundle, the notions of Zariski vector bundle and G-vector bundle coincide on affinoid varieties.
      By Theorem~\ref{Putdecompthm}, $\operatorname{Pic}$ is $\mathbb{B}^1$-invariant on smooth affinoid varieties 
      of semistable canonical reduction. But unfortunately it does not seem likely that every vector bundle over
      a smooth affinoid variety of good reduction admits a finite trivialisation by affinoids that
      are also of good reduction. (We certainly cannot achieve this by refining coverings: 
      Antoine Ducros pointed out to us that if $U$ is an affinoid of bad reduction then every finite admissible affinoid
      covering of $U$ will contain an affinoid which is again of bad reduction because
      the Shilov boundary of the Berkovich space associated to $U$ also needs to be covered.)
      Let now $X$ be a smooth affinoid variety of good reduction and $\mathcal{E}$ a line bundle over $X$ which
      does not have a finite trivialisation by smooth affinoids of good reduction.
      Denote by $\underline{X}$ the site of admissible subsets of $X$ and admissible coverings.
      The restriction of $\operatorname{Pic}$ to $\underline{X} \cap \mathcal{R}$ does not allow
      to trivialise the line bundle $\mathcal{E}$.
   \item Analogously, let $n = 1, I = \mathbb{B}^1, \tau = \mathrm{Zar}$ and $k$ discretely valued.
      Let $\mathcal{R}$ be the category of smooth affinoid varieties of the form $\operatorname{Sp} A$ with $\Vert A \Vert = \vert k \vert$ 
      and with semistable canonical reduction. The interval object $\mathbb{B}^1$ is representable in $\mathcal{R}$:
      As above, if $\operatorname{Sp} (A) \in \operatorname{Ob} ( \mathcal{R})$ is smooth, affinoid and of semistable canonical reduction, 
      then $\operatorname{Sp} (A) \times \mathbb{B}^1 = \operatorname{Sp} (A \langle T \rangle )$ shares these properties. 
      As $\Vert A \langle T \rangle \Vert = \Vert A \Vert = \vert k \vert$, the interval object $\mathbb{B}^1$ is representable.
      The Picard group is $\mathbb{B}^1$-invariant on smooth affinoid varieties 
      of semistable canonical reduction by Theorem~\ref{ThmGerritzenPic}. But again a vector bundle over
      a smooth affinoid variety of good reduction might have no finite trivialisation by sets that
      are also of good reduction.
  \end{enumerate}
 \end{Exx}

\providecommand{\bysame}{\leavevmode\hbox to3em{\hrulefill}\thinspace}
\providecommand{\MR}{\relax\ifhmode\unskip\space\fi MR }
\providecommand{\MRhref}[2]{%
  \href{http://www.ams.org/mathscinet-getitem?mr=#1}{#2}
}
\providecommand{\href}[2]{#2}

\end{document}